\documentclass[onecolumn,11pt,peerreview,draftcls]{IEEEconf}

\usepackage{graphicx}
\usepackage{amsmath}
\usepackage{amssymb,amsthm}
\usepackage{subfigure,subfigmat}
\usepackage{subfigmat}
\usepackage{epstopdf}

\usepackage{algorithm}
\usepackage{algorithmic}

\usepackage{multirow,multicol,threeparttable}
\usepackage{booktabs,cite}

\def\BibTeX{{\rm B\kern-.05em{\sc i\kern-.025em b}\kern-.08em T\kern-.1667em\lower.7ex\hbox{E}\kern-.125emX}}

\newtheorem{remark}{Remark}[section]
\newtheorem{proposition}{Proposition}[section]
\newtheorem{theorem}{Theorem}[section]

\newtheorem{lemma}{Lemma}[section]

\begin{document}

\title{On Tightly Bounding the Dubins Traveling Salesman's Optimum} 

\author{Satyanarayana Manyam\thanks{National Research Council Fellow, Air Force Research Laboratory, Dayton-Ohio, 45433.},~~
Sivakumar Rathinam\thanks{Associate Professor, Mechanical Engineering, Texas A \& M University, College Station, TX-77843, \textbf{corresponding author}:
srathinam@tamu.edu}
}

\markboth{}
{Murray and Balemi: Using the Document Class IEEEtran.cls} 


\maketitle

\begin{abstract}
The Dubins Traveling Salesman Problem (DTSP) has generated significant interest over the last decade due to its occurrence in several civil and military surveillance applications. Currently, there is no algorithm that can find an optimal solution to the problem. In addition, relaxing the motion constraints and solving the resulting Euclidean TSP (ETSP) provides the only lower bound available for the problem. However, in many problem instances, the lower bound computed by solving the ETSP is far below the cost of the feasible solutions obtained by some well-known algorithms for the DTSP. This article addresses this fundamental issue and presents the first systematic procedure for developing tight lower bounds for the DTSP.
\end{abstract}

\section{Introduction}
Given a set of targets on a plane and a constant $\rho \geq 0$, the Dubins Traveling Salesman Problem (DTSP) aims to find a
path such that each target is visited at least once, the radius of curvature of any point in the path is at least equal to $\rho$,
and the length of the path is minimal. This problem is a generalization of the Euclidean TSP (ETSP) and is NP-hard\cite{Rathinam_2007_IEEETASE,le2012dubins}. The DTSP belongs to a class of task allocation and path planning problems envisioned for a team of unmanned aerial vehicles in \cite{chandler98}. The DTSP has received significant attention in the literature \cite{Medeiros2010,Orient2014,macharet2013efficient,macharet2012data,sujit2013route,kenefic2008finding,macharet2011nonholonomic,Ozguner2005,ketan_2008,Rathinam_2007_IEEETASE,le2012dubins,shima,ma2006receding}, mainly due to its importance in unmanned vehicle applications, the simplicity of the problem statement, and its status as a hard problem to solve because it inherits features from both optimal control and combinatorial optimization.

Currently, there is no procedure for finding an optimal solution for the DTSP. Therefore, heuristics and approximation algorithms have been developed over the last decade to find feasible solutions. Tang and Ozguner \cite{Ozguner2005} present gradient-based heuristics for both single and multiple vehicle variants of the DTSP. Savla et al. \cite{ketan_2008} use an optimal solution to the ETSP to find a feasible solution for the DTSP, and they bound the cost of the feasible solution with respect to the optimal cost of the ETSP. Rathinam et al. \cite{Rathinam_2007_IEEETASE} develop an approximation algorithm for the DTSP in cases where the distance between any two targets is at least equal to $2\rho$. Ny et al. \cite{le2012dubins} develop an approximation algorithm for the DTSP in which the approximation guarantee is inversely proportional to the minimum distance between any two targets. The weakness of the approximation guarantees of these algorithms for the DTSP is due to the lack of a good lower bound, as all these algorithms essentially use the Euclidean distances between the targets to bound the cost of a feasible solution.

Other heuristics have been used for solving the DTSP. A receding horizon approach that involves finding an optimal Dubins path through three consecutive targets is used to generate feasible solutions in \cite{ma2006receding}. The heuristic in \cite{Medeiros2010} finds a feasible solution by minimizing the sum of the distances travelled by the vehicle and the sum of the changes in the heading angles at each of the targets. Macharet et al. \cite{Orient2014,macharet2012data} first obtain a tour by solving the ETSP and then select the heading angle at each target using an orientation-assignment heuristic. A multiple lookahead approach is used to find feasible solutions in \cite{sujit2013route,isaiah2015motion}. Meta-heuristics have also been developed to find feasible solutions for the DTSP in \cite{kenefic2008finding,macharet2011nonholonomic}.

\begin{figure}[htb]
\centering{}
\includegraphics[width=4in]{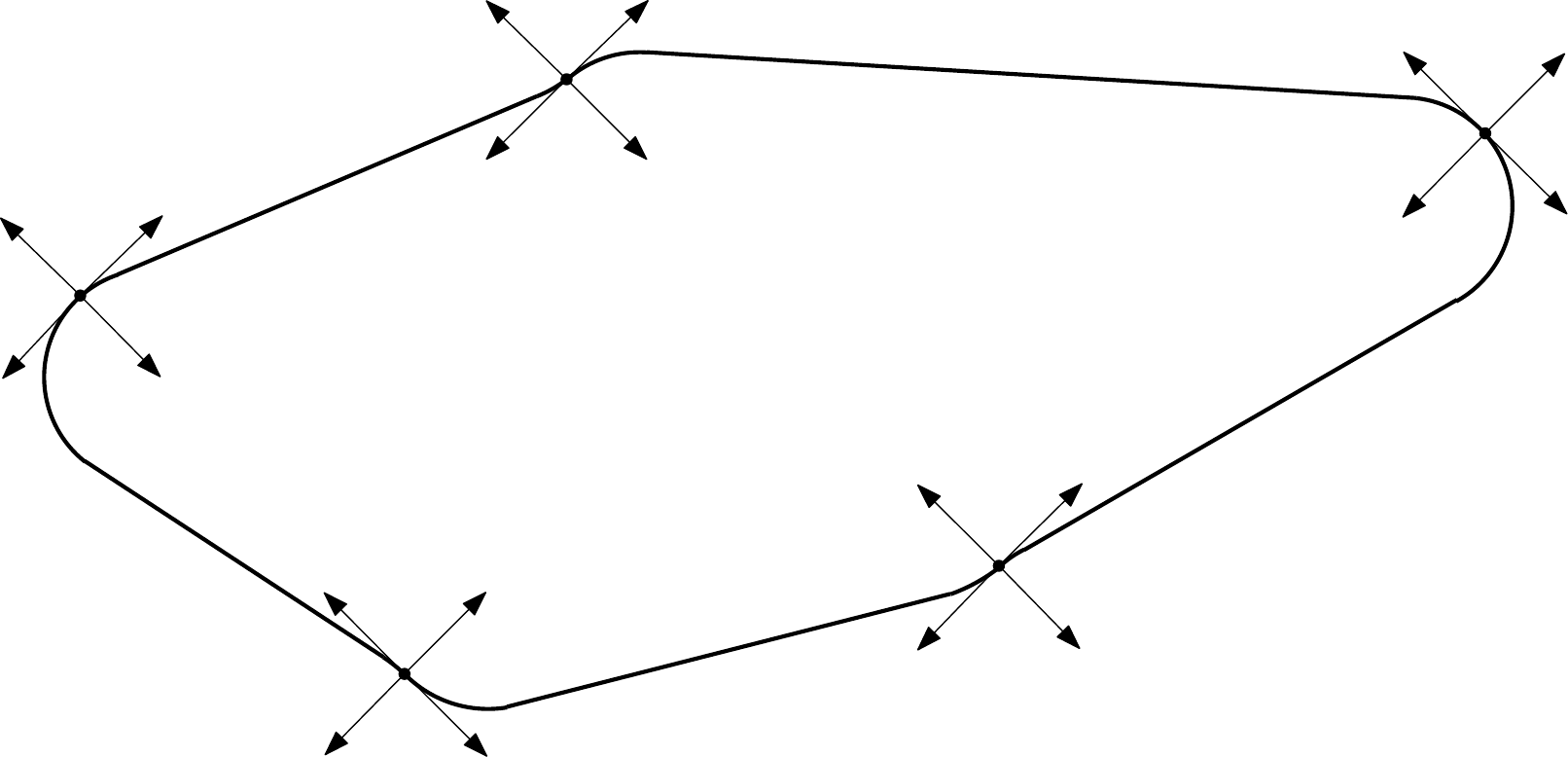}
\caption{There are four possible headings at each target. A feasible solution for the DTSP can be obtained by choosing a heading at each target and finding a corresponding optimal TSP path.}
\label{fig:ddtsp}
\end{figure}

Another common approach \cite{OberlineIEEEMagazine2010,shima} involves discretizing the heading angle at each target and posing the resulting problem as a one-in-a-set TSP (Fig. \ref{fig:ddtsp}). The greater the number of discretizations, the closer an optimal one-in-a-set TSP solution gets to the optimal DTSP solution. This approach provides a natural way to find a good, feasible solution to the problem\cite{OberlineIEEEMagazine2010}. However, this also requires us to solve a large one-in-a-set TSP, which is combinatorially hard. Nevertheless, this approach provides an upper bound \cite{shima} for the optimal cost of the DTSP, and simulation results indicate that the cost of the solutions start to converge with more than 15 discretizations at each target.

The fundamental question with regard to all the above heuristics and approximation algorithms is how close a feasible solution {\it actually is} to the optimum. For example, Fig. \ref{fig:ub_euc} shows the cost of the feasible solutions obtained by solving the one-in-a-set TSP and the ETSP for 25 instances with 20 targets in each instance. Even with 32 discretizations of the possible angles at each target, the cost of the feasible solution is at least 30\% greater than the corresponding optimal ETSP cost for several of these instances. As the optimal cost is not known for the DTSP, identifying a tight lower bound is crucial for determining the quality of the solutions that have been provided as well as for developing constant factor approximation algorithms.

\begin{figure}[!h]
\centering{}
\includegraphics[width=4in]{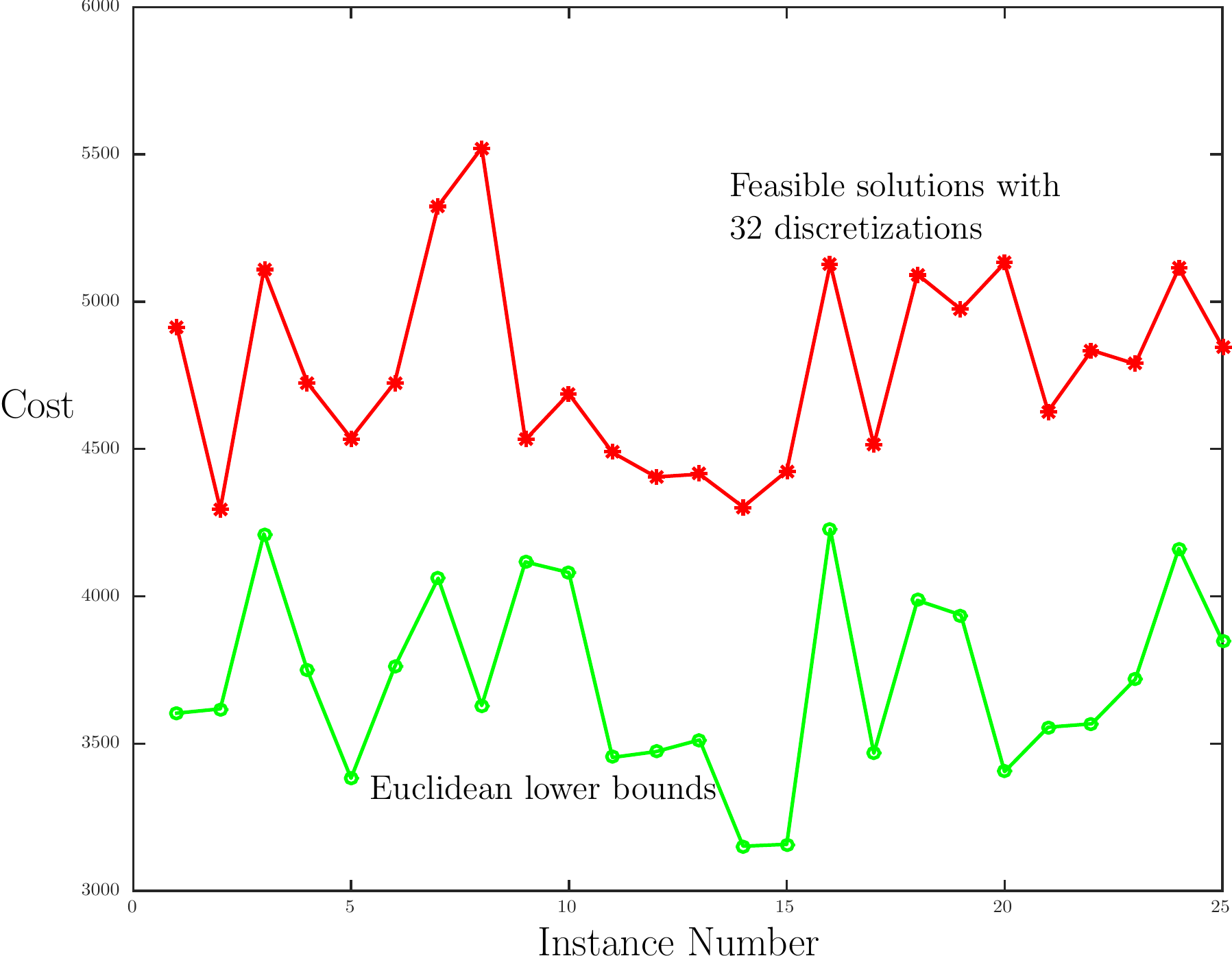}
\caption{A comparison between the cost of the feasible solution (upper bound) obtained by solving the one-in-a-set TSP with 32 discretizations and the optimal cost of the corresponding Euclidean TSP (lower bound) for 25 instances. There are 20 targets in each instance, and the location of each target is sampled from a 1000$\times$1000 square. Also, the minimum turning radius of the vehicle is set to 100.}
\label{fig:ub_euc}
\end{figure}

This fundamental question was the motivation for the bounding algorithms in \cite{manyam2015lower,manyam2015computation,manyam2012computation,manyam2013computation}. In these algorithms, the requirement that the arrival and departure angles must be equal at each target is removed, and instead there is a penalty in the objective function whenever the requirement is violated. This results in a max-min problem where the minimization problem is an asymmetric TSP (ATSP) and the cost of traveling between any two targets requires solving a new optimal control problem. In terms of lower bounding, the difficulty with this approach is that we are not currently aware of any algorithm that will guarantee a lower bound for the optimal control problem. Nonetheless, this is a useful approach, and advances in lower bounding optimal control problems will lead to finding lower bounds for the DTSP.

\begin{figure}[htb]
\centering{}
\includegraphics[width=4in]{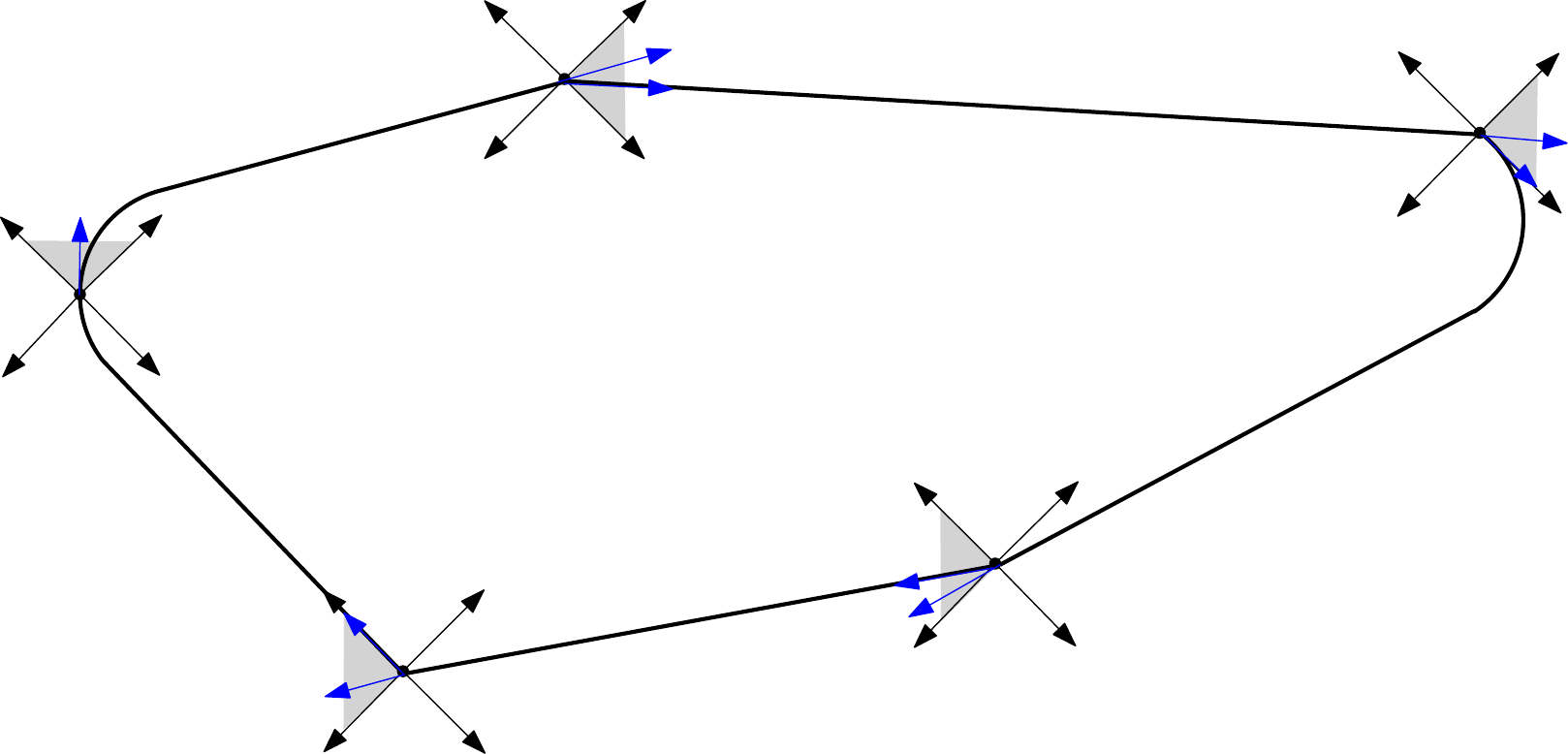}
\caption{There are four intervals at each target. A lower bound for the DTSP can be obtained by choosing an interval and restricting both the arrival and the departure angles to be in the chosen interval at each target, and then finding a corresponding optimal TSP path. The shaded interval at each target shows the chosen interval with the arrival and departure angles in blue. }
\label{fig:ldtsp}
\end{figure}

In this article, we propose a new approach to finding tight lower bounds for the DTSP. This is the first systematic procedure available for the DTSP and is a natural counterpart to the one-in-a-set TSP approach we discussed above. In this approach, we remove the requirement that the arrival angle and the departure angle at each target must be the same, but we restrain these angles so that they belong to one sector or interval (refer to Fig. \ref{fig:ldtsp}). The lower bounding problem aims to choose an interval at each target such that the arrival angle and the departure angle at the target belong to the same interval, each target is visited at least once, and the sum of the costs of travelling between the targets is minimized. The cost of traveling between two intervals corresponding to two distinct targets now reduces to a new optimal control problem, which we refer to as the \textbf{Dubins interval problem}. Given two targets and an interval at each target, the problem is to find a Dubins path such that the departure angle at the initial target and the arrival angle at the final target belong to the given intervals and the length of the path is minimal. The lower bounding problem is a one-in-a-set TSP and can be solved just like the upper bounding problem. If the size of each of the intervals at each target reduces to zero, the lower bounding problem reduces to the DTSP. If there is only one interval of size $2\pi$ at each target, the result is a Euclidean TSP. As the size of the intervals at the targets becomes smaller, the one-in-a-set TSP becomes combinatorially hard, similar to the upper bounding problem. Nevertheless, this provides a systematic approach to finding lower bounds for the DTSP, provided the Dubins interval problem can be solved.

The Dubins interval problem is a new generalization of the standard Dubins problem \cite{Dubins1957} which has not been formulated or solved\footnote{We point out that we are also working on a completely different approach which is based on optimal control in \cite{manyam2015shortest}. However, the approach in \cite{manyam2015shortest} still relies on the results of this paper. Second, unlike the analysis and results in this paper, the work in \cite{manyam2015shortest} does not provide information about the rate of change of lengths of the Dubins paths as a function of the heading angles at the targets. This information is critical and very useful in the development of bounds for feasible solutions and approximation algorithms for the DTSP.} in the literature. The difficulty with solving this problem lies in the fact that the length of the shortest Dubins paths between any two targets is a non-linear, discontinuos function of the heading angles of the targets. Therefore, finding the optimal heading angles from the given intervals at the targets that minimizes the length of the Dubins path is non-trivial. In this article, we solve the Dubins interval problem using the monotonicity properties and the extremal values of the length of the Dubins paths. \\

The following are the contributions of this article:
\begin{enumerate}
\item The formulation of the lower bounding problem for the DTSP as a novel one-in-a-set TSP where the cost of traveling between any two targets requires solving a Dubins Interval Problem. This is the first formulation that aims to provide a tight lower bound for the DTSP.
\item The first algorithm to solve the Dubins Interval Problem by exploiting its structure and monotonicity properties.
\item Numerical results to corroborate the performance of the proposed lower bounding approach for the 25 instances shown in figure \ref{fig:ub_euc}.
\end{enumerate}


\section{Lower Bounding Problem Formulation}
The set of targets is denoted by $T=\{1,2,\cdots,n\}$, where $n$ is the number of targets. The set of available angles $[0,2\pi]$ at any target $i$ is partitioned into a collection of closed intervals denoted by ${\mathcal{I}}_i := \{[0,\varphi_{i1}],[\varphi_{i1},\varphi_{i2}],\cdots,$ $[\varphi_{im_{i-1}},\varphi_{im_{i}}=2\pi]\}$, where $m_i (\geq 1)$ denotes the number of intervals at target $i$ and the $\varphi_{ij}$ are constants such that $0\leq \varphi_{i1}\leq\varphi_{i2}\leq \cdots \leq \varphi_{im_{i}}=2\pi$. Let $(x_i,y_i)$ denote the location of target $i\in T$, and let the arrival angle and the departure angle of the vehicle at target $i$ be denoted by $\theta_{ia}$ and $\theta_{id}$, respectively. The configuration of the vehicle leaving target $i$ at $\theta_{id}$ is then denoted by $(x_i,y_i,\theta_{id})$, and $(x_i,y_i,\theta_{ia})$ similarly denotes the vehicle's arrival configuration. The length of the shortest Dubins path from $(x_i,y_i,\theta_{id})$ to $(x_j,y_j,\theta_{ja})$ is denoted by $d_{ij}(\theta_{id},\theta_{ja})$. Given an interval $I_i$ at target $i$ and an interval $I_j$ at target $j$, define $d^*_{ij}(I_i,I_j):= \min_{\theta_{id}\in I_i,\theta_{ja}\in I_j} d_{ij}(\theta_{id},\theta_{ja})$. The objective of the Bounding Problem {\bf (BP)} is to find a sequence of targets $(s_1,s_2,\cdots,s_n)$, $s_i\in T$, to visit and an interval $I_{s_i}\in {\mathcal{I}}_i$ for each target $s_i\in T$ such that
\begin{itemize}
\item each target is visited at least once, and
\item the cost $\sum_{i=1}^{n-1} d^*_{s_is_{i+1}}(I_{s_i},I_{s_{i+1}}) + d^*_{s_ns_{1}}(I_{s_{n}},I_{s_1}) $ is minimized.
\end{itemize}

Addressing this {\bf BP}  first requires solving $\min_{\theta_{id}\in I_i,\theta_{ja}\in I_j} d_{ij}(\theta_{id},\theta_{ja})$. Once this problem is solved, the {\bf BP}  is essentially a one-in-a-set TSP. In this article, we transform the one-in-a-set TSP into an ATSP using the Noon-Bean transformation \cite{noon_lagrangian_1991} and then convert the resulting ATSP into a symmetric TSP using the transformation in \cite{TSPbook}. The symmetric TSP is solved using the Concorde solver \cite{Applegate:2007} to find an optimal solution.

Prior to addressing the Dubins Interval Problem in the next section, we first formally state the lower bounding result in the following proposition.

\begin{proposition}
The optimal cost to the {\bf BP} is a lower bound to the DTSP.
\end{proposition}

\begin{proof}
Any optimal solution to the DTSP is a feasible solution to the {\bf BP}  for any positive number of intervals at each target. Therefore, the optimal cost of the {\bf BP}  must be a lower bound to the optimal cost of the DTSP. 
\end{proof}

\section{Dubins Interval Problem}
Without loss of generality, let the Dubins interval problem be denoted as $\min_{\theta_{1}\in I_1,\theta_{2}\in I_2} d_{12}(\theta_{1},\theta_{2})$, where $d_{12}(\theta_{1},\theta_{2})$ indicates the shortest path (also referred to as the Dubins  path) for traveling from $(x_{1},y_{1},\theta_{1})$ to $(x_{2},y_{2},\theta_{2})$ subject to the minimum turning radius constraint (Fig. \ref{fig:dI}). Here the interval $I_k$ is defined as $[\theta^{min}_k,\theta^{max}_k]\subseteq [0,2\pi]$ for $k=1,2$. Given an initial configuration $(x_{1},y_{1},\theta_{1})$ and a final configuration $(x_{2},y_{2},\theta_{2})$, L.E. Dubins \cite{Dubins1957} showed that the shortest path for a vehicle to travel between the two configurations subject to the minimum turning radius ($\rho$) constraint must consist of at most three segments, where each segment is a circle of radius $\rho$ or a straight line. Specifically, if a curved segment of radius $\rho$ along which the vehicle travels in a counterclockwise (clockwise) rotational motion is denoted by $L(R)$, and the segment along which the vehicle travels straight is denoted by $S$, then the shortest path is one of $RSR$, $RSL$, $LSR$, $LSL$, $RLR$, and $LRL$.

\begin{figure}[htb]
\centering{}
\includegraphics[scale=1]{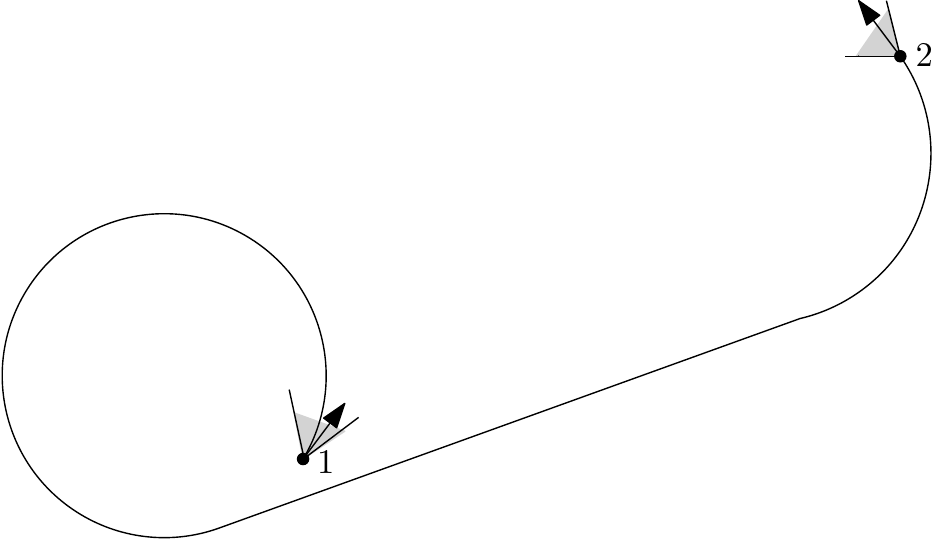}
\caption{A feasible solution to the Dubins interval problem.}
\label{fig:dI}
\end{figure}

Let $RSL(\theta_{1},\theta_{2})$ denote the length of the $RSL$ path from $(x_{1},y_{1},\theta_{1})$ to $(x_{2},y_{2},\theta_{2})$. $RSL(\theta_1,\theta_2)$ is set to $\infty$ if the $RSL$ path does not exist. Let $RSR(\theta_{1},\theta_{2})$, $LSR(\theta_{1},\theta_{2})$, $LSL(\theta_{1},\theta_{2})$, $RLR(\theta_{1},\theta_{2})$, and $LRL(\theta_{1},\theta_{2})$ be defined in a similar way. Using these definitions, the Dubins interval problem can be written as follows:

\begin{align}\label{main:problem}
\min_{\theta_{1}\in I_1,\theta_{2}\in I_2} d_{12}(\theta_{1},\theta_{2}) &= \min_{\theta_{1}\in I_1,\theta_{2}\in I_2}\{RSR(\theta_{1},\theta_{2}),RSL(\theta_{1},\theta_{2}), LSR(\theta_{1},\theta_{2}), LSL(\theta_{1},\theta_{2}), RLR(\theta_{1},\theta_{2}),LRL(\theta_{1},\theta_{2})\}.
\end{align}

\begin{remark}
 $d_{12}(\theta_{1},\theta_{2})$ is a lower semicontinuous function and is minimized over closed and bounded intervals $I_1$ and $I_2$. Therefore, the Dubins interval problem is well defined, $i.e.$, there exist $\theta_1^*\in I_1$ and $\theta_2^* \in I_2$ such that  $d_{12}(\theta_{1}^*,\theta_{2}^*)=  \min_{\theta_{1}\in I_1,\theta_{2}\in I_2} d_{12}(\theta_{1},\theta_{2})$.
\end{remark}

To solve the Dubins interval problem, we also consider shortest paths that contain at most two segments between $(x_1,y_1)$ and $(x_2,y_2)$. For any path ${\mathcal{T}}\in \{RS, LS, SR, SL, RL, LR\}$ and $\theta_1\in I_1$, let ${\mathcal{T}}^1(\theta_1)$ denote the distance of the shortest path of type ${\mathcal{T}}$ that starts at $(x_1,y_1)$ with a departure angle of $\theta_1$ and arrives at $(x_2,y_2)$ with an arrival angle in $I_2$. In this case, the arrival angle at $(x_2,y_2)$ will be a function of $\theta_1$ and ${\mathcal{T}}$ and is denoted as $\theta_2({\mathcal{T}},\theta_1)$. ${\mathcal{T}}^1(\theta_1)$ is set to $\infty$ if a path of type ${\mathcal{T}}$ does not exist or if $\theta_2({\mathcal{T}},\theta_1) \notin I_2$.
 Similarly, let ${\mathcal{T}}^2(\theta_2)$ denote the distance of the shortest path of type ${\mathcal{T}}$ that starts at $(x_1,y_1)$ with a departure angle in $I_1$ and arrives at $(x_2,y_2)$ with an arrival angle of $\theta_2$. In this case, the departure angle at $(x_1,y_1)$ will be a function of $\theta_2$ and ${\mathcal{T}}$ and is denoted as $\theta_1({\mathcal{T}},\theta_2)$. ${\mathcal{T}}^2(\theta_2)$ is set to $\infty$ if the path of type ${\mathcal{T}}$ does not exist or if $\theta_1({\mathcal{T}},\theta_2) \notin I_1$. From the definitions, note that $\min_{\theta_1\in I_1}{\mathcal{T}}^1(\theta_1) = \min_{\theta_2\in I_2}{\mathcal{T}}^2(\theta_2) $. \\


The following theorem provides a way to further simplify equation \eqref{main:problem} and solve the Dubins interval problem:

\begin{theorem}\label{theorem:main}
\begin{align}
\min_{\theta_1 \in I_1} \min_{\theta_{2}\in I_2}\{d_{12}(\theta_{1},\theta_{2})\}  & = \min\{d^*, \min_{\theta_{1} \in I_1}\{RS^1(\theta_{1}),SR^1(\theta_{1}),LS^1(\theta_{1}),SL^1(\theta_{1}),LR^1(\theta_{1}),RL^1(\theta_{1})\}\}
\end{align}

where \[d^*: = \min\{d_{12}(\theta_{1}^{min},\theta_{2}^{min}),d_{12}(\theta_{1}^{max},\theta_{2}^{min}),d_{12}(\theta_{1}^{min},\theta_{2}^{max}),d_{12}(\theta_{1}^{max},\theta_{2}^{max})\}.\]
\end{theorem}

In English, this theorem states that an optimal path to the Dubins interval problem must be one of the following:
\begin{enumerate}
\item An optimal Dubins path consisting of at most three segments such that both the arrival and departure angles at each target belong to one of the boundary values of the respective intervals, or
\item An optimal Dubins path consisting of at most two segments such that the angle constraints are satisfied.
\end{enumerate}

After proving this theorem in the next subsection, we will provide algorithms to solve for the optimal Dubins paths with at most two segments (i.e., to solve $\min_{\theta_{1} \in I_1} {\mathcal{P}}(\theta_{1})$ for any path ${\mathcal{P}}\in\{RS^1, SR^1, LS^1, SL^1, LR^1, RL^1\}$). As $d^*$ in the above theorem can already be computed directly using Dubins's result\cite{Dubins1957}, one can then compute the optimal cost for the Dubins interval problem and the corresponding departure and arrival angles.

\subsection{Proof of Theorem \ref{theorem:main}}

This theorem will be proved in two parts. First, we will first show how to simplify $\min_{\theta_{1}\in I_1,\theta_{2}\in I_2} {\mathcal{P}}(\theta_{1},\theta_{2})$ for any path ${\mathcal{P}}\in\{RSR, RSL, LSR, LSL\}$. Then we will address the $LRL$ and the $RLR$ paths.\\

\subsubsection{\textbf{Optimizing RSR, RSL, LSR, and LSL paths}}\label{subsec:CSC}

The following result is known \cite{xavier} for each of the paths ${\mathcal{P}} \in \{RSR, RSL, LSR, LSL\}$ from $(x_{1},y_{1},\theta_{1})$ to $(x_{2},y_{2},\theta_{2})$: \\
\begin{lemma}\label{lem:csc}
For any ${\mathcal{P}} \in \{RSR, RSL, LSR, LSL\}$ and $i=1,2$, either $\frac{\partial {\mathcal{P}}(\theta_1,\theta_2)}{\partial \theta_i} \geq 0$ $\forall ~\theta_i$ or $\frac{\partial {\mathcal{P}}(\theta_1,\theta_2)}{\partial \theta_i} \leq 0$ $\forall ~\theta_i$ when $\mathcal{P}$ exists and none of its curved segments vanish. \\
\end{lemma}

\begin{figure}[htb]
\centering{}
\includegraphics[scale=.6]{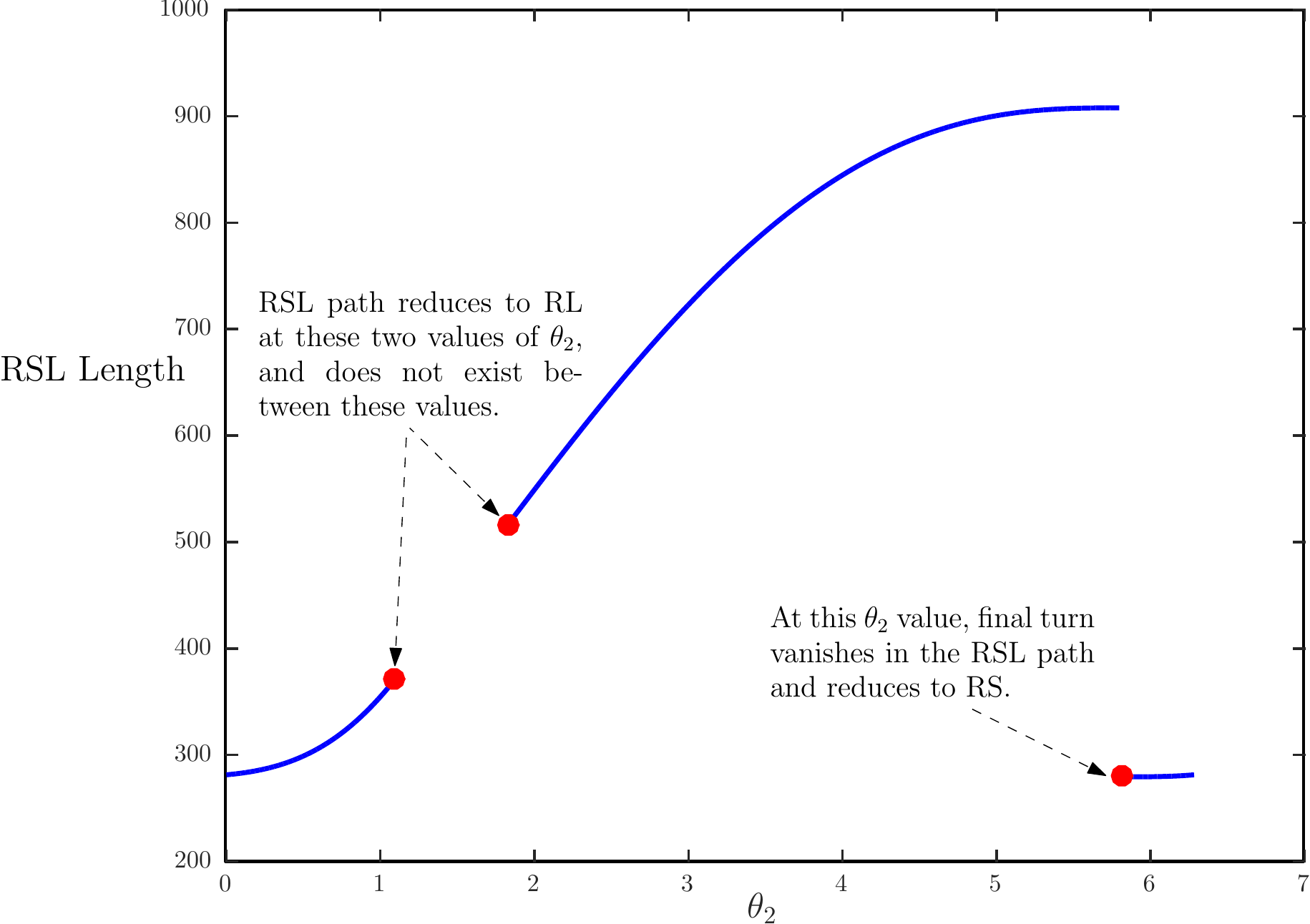}
\caption{Given $\theta_1$, the length of the RSL path varies monotonically with respect to $\theta_2$ wherever the path exists and none of its curved segments vanish.}
\label{fig:RSL}
\end{figure}

Now let us apply the above lemma to the $RSL$ path. The $RSL$ path ceases to exist when the segment $S$ vanishes, i.e., the $RSL$ path reduces to an $RL$ path. In addition, when one of the curved segments vanishes, the $RSL$ path reduces to either the $RS$ or the $SL$ path (refer to Fig. \ref{fig:RSL}). Therefore, given $\theta_1$, the optimum for $\min_{\theta_2 \in [\theta^{min}_2,\theta^{max}_2]}RSL(\theta_{1},\theta_{2})$ must be attained when $\theta_2=\theta^{min}_2$ or $\theta_2=\theta^{max}_2$ or when the $RSL$ path reduces to an $RL$, $RS$, or  $SL$ path. This can be stated as follows:

\begin{align}
\min_{\theta_{2}\in I_2}\{RSL(\theta_{1},\theta_{2})\} & := \min\{RSL(\theta_{1},\theta_{2}^{min}),RSL(\theta_{1},\theta_{2}^{max}), RS^1(\theta_{1}),SL^1(\theta_{1}),RL^1(\theta_{1})\}.
\end{align}

Therefore,
\begin{align}
& \min_{\theta_{1} \in I_1}\min_{\theta_{2}\in I_2}\{RSL(\theta_{1},\theta_{2})\} \nonumber \\ & = \min_{\theta_{1}\in I_1} \min\{RSL(\theta_{1},\theta_{2}^{min}),RSL(\theta_{1},\theta_{2}^{max}), RS^1(\theta_{1}),SL^1(\theta_{1}),RL^1(\theta_{1})\} \nonumber \\
& =  \min\{\min_{\theta_{1}\in I_1}RSL(\theta_{1},\theta_{2}^{min}),\min_{\theta_{1}\in I_1} RSL(\theta_{1},\theta_{2}^{max}), \min_{\theta_{1}\in I_1}\{RS^1(\theta_{1}),SL^1(\theta_{1}),RL^1(\theta_{1})\}\}. \label{eq:minRSL}
\end{align}

Similarly, using Lemma \ref{lem:csc} again, we get the following:
\begin{align}
\min_{\theta_{1} \in I_1}RSL(\theta_{1},\theta_{2}^{min}) & = \min\{RSL(\theta_{1}^{min},\theta_{2}^{min}),RSL(\theta_{1}^{max},\theta_{2}^{min}), RS^2(\theta_{2}^{min}),SL^2(\theta_{2}^{min}),RL^2(\theta_{2}^{min})\}, \label{eq:rsl_min} \\
\min_{\theta_{1} \in I_1}RSL(\theta_{1},\theta_{2}^{max}) & = \min\{RSL(\theta_{1}^{min},\theta_{2}^{max}),RSL(\theta_{1}^{max},\theta_{2}^{max}), RS^2(\theta_{2}^{max}),SL^2(\theta_{2}^{max}),RL^2(\theta_{2}^{max})\}. \label{eq:rsl_max}
\end{align}

Now, one can easily verify the following:
\begin{align}
 \text{For any }{\mathcal{T}}\in\{RS,SL,RL\}, ~\min_{\theta_1\in I_1}{\mathcal{T}} ^1(\theta_1) \leq {\mathcal{T}} ^2(\theta_2^{min})\text{ and } \min_{\theta_1\in I_1}{\mathcal{T}} ^1(\theta_1) \leq {\mathcal{T}} ^2(\theta_2^{max}). \label{eq:popt}
\end{align}

Substituting for $\min_{\theta_{1} \in I_1}RSL(\theta_{1},\theta_{2}^{min})$ and $\min_{\theta_{1} \in I_1}RSL(\theta_{1},\theta_{2}^{max})$ in \eqref{eq:minRSL} using equations \eqref{eq:rsl_min} and \eqref{eq:rsl_max} and simplifying further using \eqref{eq:popt}, we get

\begin{align}\label{main:RSL}
\min_{\theta_1 \in I_1} \min_{\theta_{2}\in I_2} RSL(\theta_{1},\theta_{2})  & = \min\{RSL^*, \min_{\theta_{1} \in I_1}\{RS^1(\theta_{1}),SL^1(\theta_{1}),RL^1(\theta_{1})\}\},
\end{align}

where \[RSL^*: = \min\{RSL(\theta_{1}^{min},\theta_{2}^{min}),RSL(\theta_{1}^{max},\theta_{2}^{min}),RSL(\theta_{1}^{min},\theta_{2}^{max}),RSL(\theta_{1}^{max},\theta_{2}^{max})\}.\]

As Lemma \ref{lem:csc} is also applicable to $RSR$, $LSL$, and $LSR$ paths, one can use the above procedure and simplify $\min_{\theta_{1}\in I_1,\theta_{2}\in I_2} {RSR}(\theta_{1},\theta_{2})$, $\min_{\theta_{1}\in I_1,\theta_{2}\in I_2} {LSL}(\theta_{1},\theta_{2})$, and $\min_{\theta_{1}\in I_1,\theta_{2}\in I_2} {LSR}(\theta_{1},\theta_{2})$ in a similar way. Combining all these results, we obtain the following:

\begin{align}\label{main:CSC}
\min_{\theta_1 \in I_1} \min_{\theta_{2}\in I_2} \min_{{\mathcal{P}} \in \{RSR, RSL, LSR, LSL\}} {\mathcal{P}}(\theta_{1},\theta_{2})  & = \min\{\mathcal{P}^*, \min_{\theta_{1} \in I_1}\{RS^1(\theta_{1}),SR^1(\theta_{1}),LS^1(\theta_{1}),SL^1(\theta_{1}),LR^1(\theta_{1}),RL^1(\theta_{1})\}\}
\end{align}

where \[{\mathcal{P}}^*: =  \min_{{\mathcal{P}} \in \{RSR, RSL, LSR, LSL\}} \min\{{\mathcal{P}}(\theta_{1}^{min},\theta_{2}^{min}),{\mathcal{P}}(\theta_{1}^{max},\theta_{2}^{min}),{\mathcal{P}}(\theta_{1}^{min},\theta_{2}^{max}),{\mathcal{P}}(\theta_{1}^{max},\theta_{2}^{max})\}.\] \\

\subsubsection{\textbf{Optimizing RLR and LRL paths}}

Xavier et al. \cite{xavier} have shown that the $RLR$ and $LRL$ paths cannot lead to an optimal Dubins path if the distance between the two targets is greater than $4\rho$. Therefore, in this section, we assume that the distance between the two targets is at most $4\rho$. We will focus on $\min_{\theta_1\in I_1, \theta_2\in I_2} LRL(\theta_1,\theta_2)$; $\min_{\theta_1\in I_1, \theta_2\in I_2} RLR(\theta_1,\theta_2)$ can be solved in a similar way. Given $\theta_1$, unlike the length of the $RSL$ path, $LRL(\theta_1,\theta_2)$ is not monotonous with respect to $\theta_{2}$ when $LRL$ exists. Without loss of generality, we assume that $\theta_1$ = 0 and first aim to understand $LRL(0,\theta_2)$ as a function of $\theta_2$ (refer to Fig. \ref{fig:LRL_proof}). Target 1 is located at the origin and target 2 is located at $(\bar{x},\bar{y})$. The angles $\alpha$ and $\beta$ in Fig. \ref{fig:LRL_proof} are functions of $\theta_2$. For brevity, we use $\alpha$ and $\beta$ in place of $\alpha(\theta_2)$ and $\beta(\theta_2)$, respectively.  Let $LRL(0,\theta_2)$ be denoted as ${\mathfrak{D}}(\theta_2):=(2\pi+ 2\alpha + 2\beta + \theta_2)\rho$. In the ensuing discussion, we use the fact that the length of the $R$ segment in an $LRL$ path must be greater than $\pi\rho$ (i.e., $0<\alpha+\beta<\pi$) for the $LRL$ path to be an optimal Dubins path between any two targets \cite{Dubins1957,boissonat}.

\begin{figure}[htb]
\centering{}
\includegraphics[scale=1]{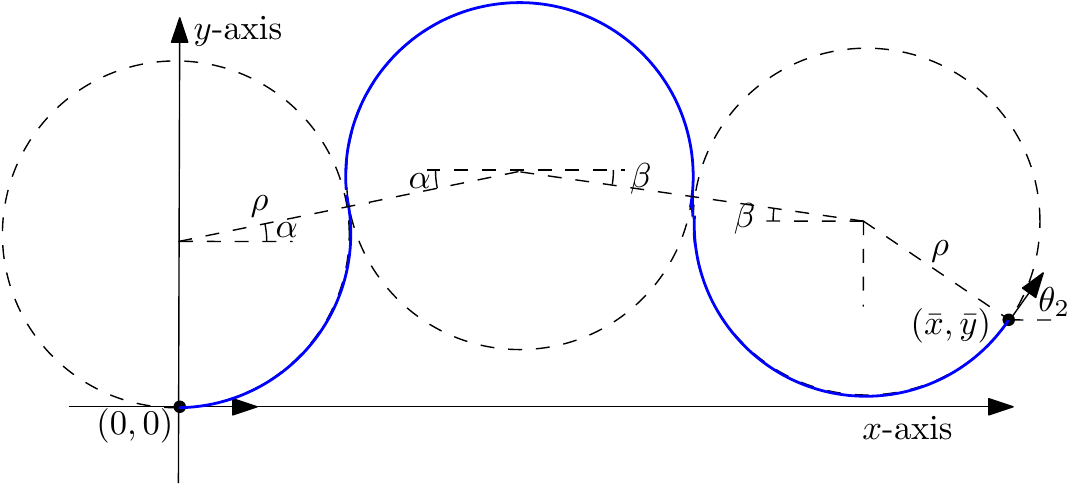}
\caption{$LRL$ path for $\theta_1=0$.}
\label{fig:LRL_proof}
\end{figure}

\begin{lemma}\label{lem:cc}
If the $LRL$ path exists and none of its curved segments vanish, then for any $\theta_2$ such that $0<\alpha(\theta_2)+\beta(\theta_2)<\pi$, $\frac{{\mathfrak{dD}}}{d\theta_2} \neq 0$ except when ${\mathfrak{D}}(\theta_2)$ reaches a maximum, when $\theta_2$ satisfies $\alpha+\pi/2 =\theta_2$.
\end{lemma}

\begin{proof}
Using Fig. \ref{fig:LRL_proof}, $\alpha$ and $\beta$ can be obtained in terms of $\theta_2$ as follows:

\begin{align}
& 2\rho \sin \alpha + \rho  = 2\rho \sin \beta + \rho \cos \theta_2 + \bar{y}, \\
& 2\rho \cos \alpha + 2 \rho \cos \beta + \rho \sin \theta_2 = \bar{x}.
\end{align}

Differentiating and simplifying the above equations, we get

\begin{align}
\cos \alpha \frac{d \alpha}{d \theta_2} -\cos \beta \frac{d \beta}{d \theta_2} &= -\frac{\sin \theta_2}{2}, \\
\sin \alpha \frac{d \alpha}{d \theta_2} +\sin \beta \frac{d \beta}{d \theta_2} &= \frac{\cos \theta_2}{2}.
\end{align}

Further solving for the derivatives, we get
\begin{align}
\frac{d \beta}{d \theta_2} &= \frac{\cos (\theta_2-\alpha)}{2 \sin (\alpha + \beta)}, \\
\frac{d \alpha}{d \theta_2} &= \frac{\cos (\theta_2+\beta)}{2 \sin (\alpha + \beta)}.
\end{align}

Therefore,
\begin{align}
\frac{{d\mathfrak{D}}}{d\theta_2} & = \rho (2\frac{d \beta}{d \theta_2} + 2\frac{d \alpha}{d \theta_2} + 1) \\
& = \rho(\frac{\cos (\theta_2-\alpha)}{ \sin (\alpha + \beta)} + \frac{\cos (\theta_2+\beta)}{ \sin (\alpha + \beta)} +1 ).
\end{align}
Equation $\frac{{d\mathfrak{D}}}{d\theta_2}=0$ yields the following possibilities: $\theta_2=\frac{\pi}{2} + \alpha$ or $\theta_2 + \beta = -\frac{\pi}{2}$. $\theta_2 + \beta = -\frac{\pi}{2}$ corresponds to the case where the second left turn disappears; there is a jump in the length of the $LRL$ path at this $\theta_2$, and therefore $\frac{{d\mathfrak{D}}}{d\theta_2}$ does not exist. $\theta_2=\frac{\pi}{2} + \alpha$ corresponds to the case where the turn angle in the right turn is equal to the turn angle in the second left turn; one can verify that ${\mathfrak{D}}(\theta_2)$ reaches a maximum at this point because $\frac{{d^2\mathfrak{D}}}{d\theta_2^2}=-\frac{3\rho}{2} \frac{1+\cos (\alpha + \beta)}{\sin (\alpha + \beta)}< 0$ (refer to Fig. \ref{fig:LRL}).

\end{proof}

\begin{figure}[htb]
\centering{}
\includegraphics[scale=.6]{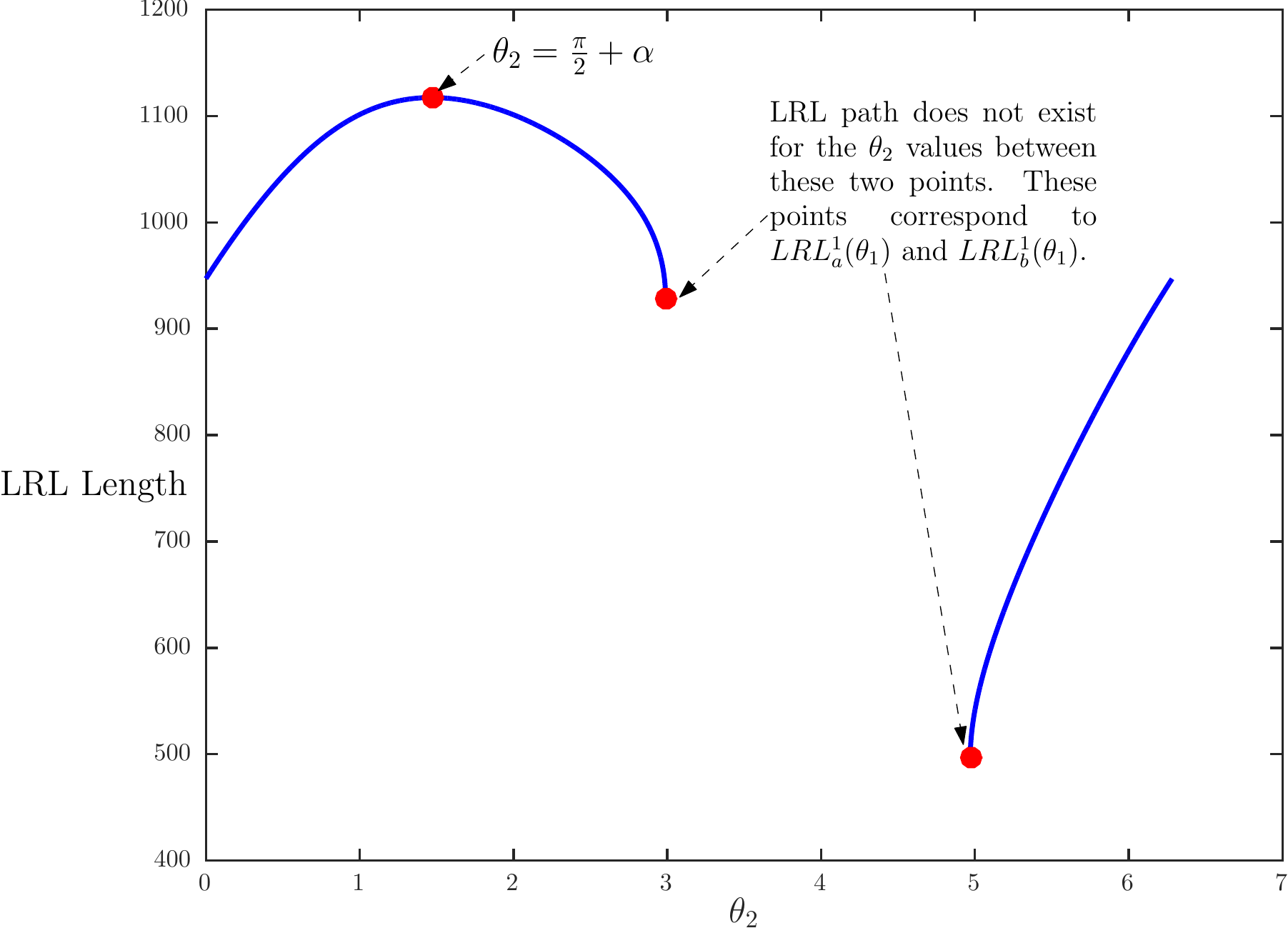}
\caption{Given $\theta_1$, the length of the $LRL$ path reaches a maximum when $\theta_2=\frac{\pi}{2} + \alpha$, as shown. This figure also shows the values of $\theta_2$ where the $LRL$ path just ceases to exist.}
\label{fig:LRL}
\end{figure}

\begin{figure}[htb]
\centering{}
\includegraphics[scale=1]{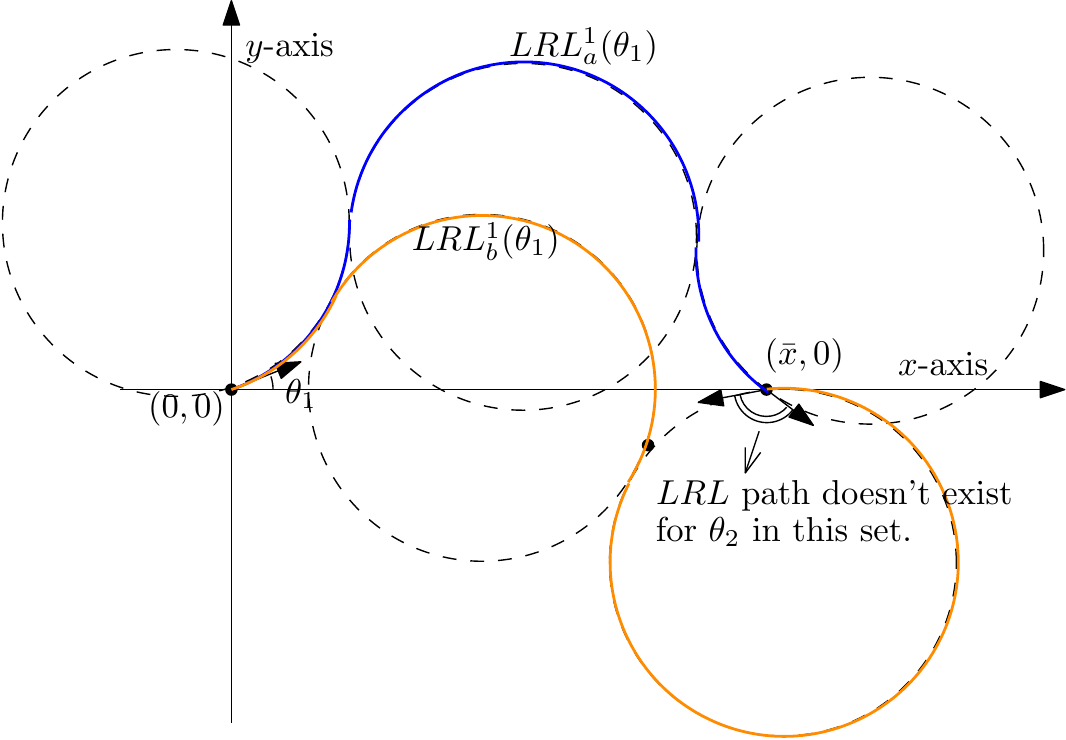}
\caption{Given $\theta_1$, the $LRL$ paths when the arc angle in the right turn is $\pi$. This figure shows the angles for $\theta_2$ when the $LRL$ path does not exist.}
\label{fig:LRL_exist}
\end{figure}

The derivatives of $LRL(\theta_1,\theta_2)$ do not exist when any turn in the path disappears or when the angle in the right turn becomes equal to $\pi$, as shown in Fig. \ref{fig:LRL_exist}. The length of the two paths (Fig. \ref{fig:LRL_exist}) when the $LRL$ path just ceases to exist are denoted by $LRL^1_a(\theta_1)$ and $LRL^1_b(\theta_1)$. Therefore, applying the above lemma to the $LRL$ path and following similar steps to those in subsection \ref{subsec:CSC}, we get the following result:

\begin{align}
\min_{\theta_{2}\in I_2}\{LRL(\theta_{1},\theta_{2})\} & := \min\{LRL(\theta_{1},\theta_{2}^{min}),LRL(\theta_{1},\theta_{2}^{max}), LR^1(\theta_{1}),RL^1(\theta_{1}),LRL^1_a(\theta_{1}),LRL^1_b(\theta_{1})\}.
\end{align}

Again, as in subsection \ref{subsec:CSC}, one can further simplify the above optimization problem:
\begin{align}
\min_{\theta_1 \in I_1} \min_{\theta_{2}\in I_2}\{LRL(\theta_{1},\theta_{2})\}  & = \min\{LRL^*, \min_{\theta_{1} \in I_1}\{LR^1(\theta_{1}),RL^1(\theta_{1}),LRL^1_a(\theta_{1}),LRL^1_b(\theta_{1})\}\} \label{main:LRL}
\end{align}

where \[LRL^*: = \min\{LRL(\theta_{1}^{min},\theta_{2}^{min}),LRL(\theta_{1}^{max},\theta_{2}^{min}),LRL(\theta_{1}^{min},\theta_{2}^{max}),LRL(\theta_{1}^{max},\theta_{2}^{max})\}.\]

Note that $LRL^1_a(\theta_1)$ and $LRL^1_b({\theta_1})$ can never result in an optimal Dubins path because the angle in the right turn is equal to $\pi$\cite{boissonat}. Therefore, once equation \eqref{main:LRL} is substituted in equation \eqref{main:problem}, the functions $LRL^1_a(\theta_1)$ and $LRL^1_b({\theta_1})$ will drop out.

$\min_{\theta_1 \in I_1} \min_{\theta_{2}\in I_2}\{RLR(\theta_{1},\theta_{2})\}$ can be simplified in a similar way. Hence, combining the above results with equation \eqref{main:CSC}, we obtain the result stated in Theorem \ref{theorem:main}.

\section{Algorithms for Optimizing Dubins Paths with At Most Two Segments}

\begin{figure}[htb]
\centering{}
\includegraphics[scale=1]{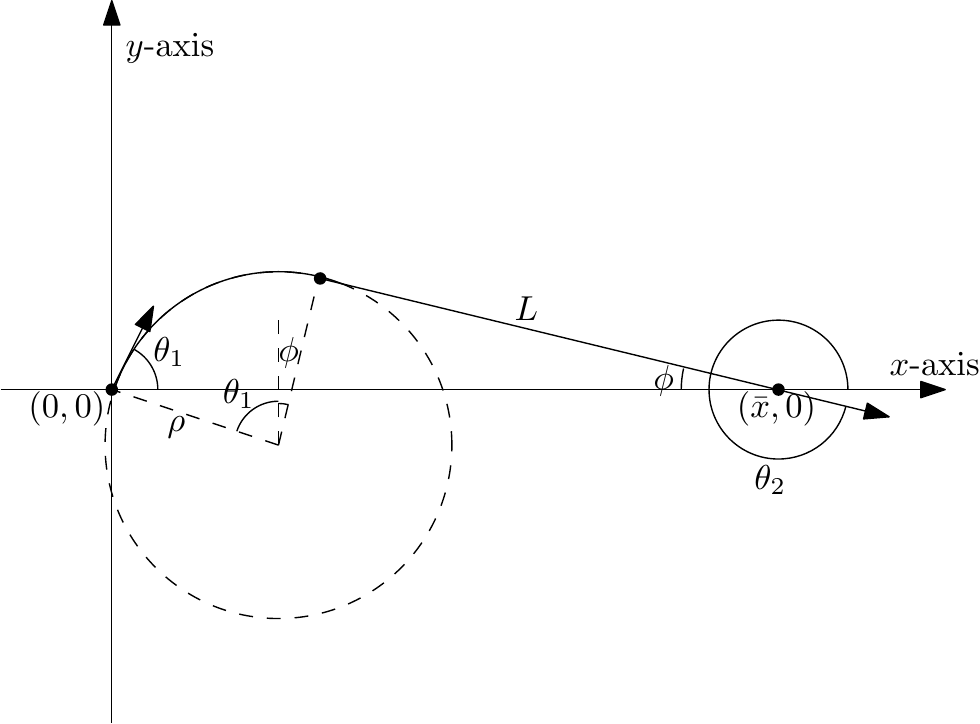}
\caption{$RS$ path}
\label{fig:RSpath}
\end{figure}

\begin{figure}[htb]
\centering{}
\includegraphics[scale=1]{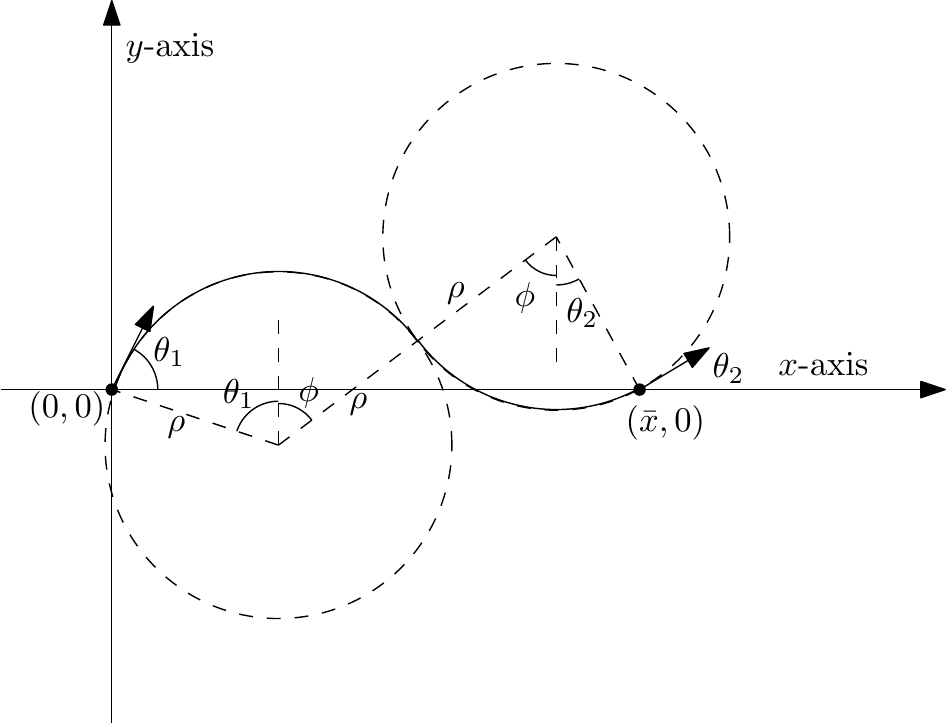}
\caption{$RL$ path}
\label{fig:RLpath}
\end{figure}

The only remaining step needed to solve the Dubins interval problem is to show how to optimize $\min_{\theta_{1} \in I_1}\{RS^1(\theta_{1}),SR^1(\theta_{1}),LS^1(\theta_{1}),SL^1(\theta_{1}),$ $LR^1(\theta_{1}),RL^1(\theta_{1})\}$. In this section, we will solve two problems: $\min_{\theta_{1} \in I_1}\{RS^1(\theta_{1})\}$ and $\min_{\theta_{1} \in I_1}\{RL^1(\theta_{1})\}$. The remaining paths can be solved using simple transformations (reflections about the $x$ or the $y$ axis).

\subsection{Optimizing the RS path}\label{sec:RS}

Without loss of generality, a reference frame can be chosen such that target 1 is at the origin and target 2 lies on the $x$-axis as shown in Fig. \ref{fig:RSpath}. Here $\bar{x}$ represents the Euclidean distance between the targets. Given $\theta_1$, the existence of the $RS$ path as well as its length can be determined using geometry. The length of the $S$ path, the angle between the $x$-axis and the $S$ path, and the final arrival angle at target 2 are also functions of $\theta_1$ and can be expressed as $L(\theta_1)$, $\phi(\theta_1)$, and $\theta_2(RS,\theta_1)$, respectively. Let the length of the $RS$ path be denoted as ${\mathfrak{D}}(\theta_1)$. For brevity, in some places we will use $L,\phi,\theta_2$, and ${\mathfrak{D}}$ instead of $L(\theta_1), \phi(\theta_1),\theta_{2}(RS,\theta_1)$, and ${\mathfrak{D}}(\theta_1)$, respectively. Let $d_S:= \bar{x}$ if the angle of the straight line joining the two targets lies in the intervals $I_1$ and $I_2$. If the angle constraints are not satisfied, $d_S$ is set to $\infty$. Similarly, let $d_R$ denote the length of the shortest circular arc of type $R$ that joins the two targets such that the boundary angles of the arc belong to the respective intervals at the targets. If such an arc does not exist, $d_R$ is set to $\infty$. In the following lemma, we assume that $[\theta_1^{min},\theta_1^{max}]\subseteq [0,2\pi]$ and $[\theta_2^{min},\theta_2^{max}]\subseteq [0,2\pi]$.

\begin{lemma}\label{lemma:RS}
$\min_{\theta_{1} \in I_1}\{RS^1(\theta_{1})\}:=\min \{ d_S, d_R, RS^1(\theta_1^{min}), RS^2(\theta_2^{min}), RS^2(\theta_2^{max})\}$.
\end{lemma}

\begin{proof}
Refer to the appendix for the proof.
\end{proof}

\subsection{Optimizing the RL path}\label{sec:RL}

We use similar notations as in previous subsections (refer to Fig. \ref{fig:RLpath}). The angles $\phi(\theta_1)$ and $\theta_2(RL,\theta_1)$ are also written as $\phi$ and $\theta_2$, for brevity. The length of the $RL$ path is denoted as ${\mathfrak{D}}(\theta_1)$ and is equal to $\rho(\theta_1+\theta_2+2\phi)$. $RL$ paths do not exist when $\bar{x}>4\rho$. In addition, even when $0\leq \bar{x} \leq 4\rho$, there are a subset of angles of $\theta_1$ for which an $RL$ path does not exist. Moreover, given $\theta_1$, there are two possible $RL$ paths, as either $\phi+\theta_2 \leq \pi$ or  $\phi+\theta_2 > \pi$. In the following discussion and in Fig. \ref{fig:RLpath}, we assume that $\phi+\theta_2<\pi$. The other $RL$ path can be addressed similarly.

We first define some values of $\theta_1$ where the optimum can occur (these correspond to the extreme values of ${\mathfrak{D}}$ and $\theta_2$ for the RL path and will be derived later in the proof). Let $\theta^{1*}$ be the solution to the equation $\theta_2(RL,\theta_1) = \theta_1$. Also, let $\theta^{2*}$ and $\theta^{3*}$ be the solutions to equation $\phi(\theta_1) + \theta_2(RL,\theta_1)= \pi$. Let $d_L$ denote the length of the shortest circular arc of type $L$ that joins the two targets such that the boundary angles of the arc belong to the corresponding intervals at the targets and $\phi + \theta_2 < \pi$. If such an arc does not exist, then $d_L$ is set to $\infty$. Let $RL^* = \min \{RL^1(\theta_1^{max}), RL^1(\theta_1^{min}),RL^2(\theta_2^{min}),RL^2(\theta_2^{max}\}$.

\begin{lemma}\label{lemma:RL}
If $\bar{x} > 2\rho$, $\min_{\theta_{1} \in I_1}\{RL^1(\theta_{1})\}:=\min \{RL^1(\theta^{1*}), RL^1(\theta^{2*}), RL^1(\theta^{3*}), RL^*\}$. If $0\leq \bar{x}\leq 2\rho$, $\min_{\theta_{1} \in I_1}\{RL^1(\theta_{1})\}:=\min \{d_L,d_R,RL^1(\theta^{1*}), RL^*)\}$.
\end{lemma}

\begin{proof}
Refer to the appendix.
\end{proof}

\section{Numerical results}

Computational results are presented for 25 instances with 20 targets in each instance. The locations of the targets were sampled from a $1000\times1000$ square. The minimum turning radius of the vehicle was chosen to be 100. The heading angles at each target are discretized into 4, 8, 16, and 32 intervals. We use the Noon-Bean transformation to first convert the one-in-a-set TSP into an ATSP. Then we use a transformation method outlined in \cite{TSPbook} to convert the ATSP into a symmetric TSP. This method converts an asymmetric instance with $n$ nodes into a symmetric instance with $3n$ nodes. We chose this method primarily because unlike other transformations, there is no big-$M$ constant involved, and therefore we did not have any numerical difficulties such as those faced in \cite{OberlineIEEEMagazine2010,manyam2012computation,manyam2013computation}. For example, the transformed TSP instance with 32 discretizations at each target has 1920 nodes. Each of the transformed TSP instances was solved to optimality using the CONCORDE solver\cite{Applegate:2007}. The improvement of the lower bounds as the number of discretizations or intervals increases is shown in Fig. \ref{fig:lb_prog}. On average, the  improvement of the lower bounds with respective to the optimal ETSP cost for 32 intervals was 22.28\%.

A feasible solution was also obtained by discretizing the angles at each target (32 values) and applying the above transformation procedure. The comparison of the cost of the feasible solution with respect to the optimal Euclidean TSP cost and the lower bound (corresponding to 32 intervals at each target) for the 25 instances is shown in Fig. \ref{fig:best_ublb}. The average deviation of the cost of the feasible solution from its corresponding lower bound is 5.2\%, while the average deviation of the cost of the feasible solution from its corresponding ETSP cost is 29.2\%. In one of the instances, we found the cost of the feasible solution from its corresponding lower bound improved by approximately 44\%. These results show that the proposed approach can be used to obtain tight lower bounds for the DTSP. A feasible DTSP solution and an optimal solution corresponding to the lower bound for an instance are shown in Fig. \ref{fig:paths}.

\begin{figure}[!h]
\centering{}
\includegraphics[width=4in]{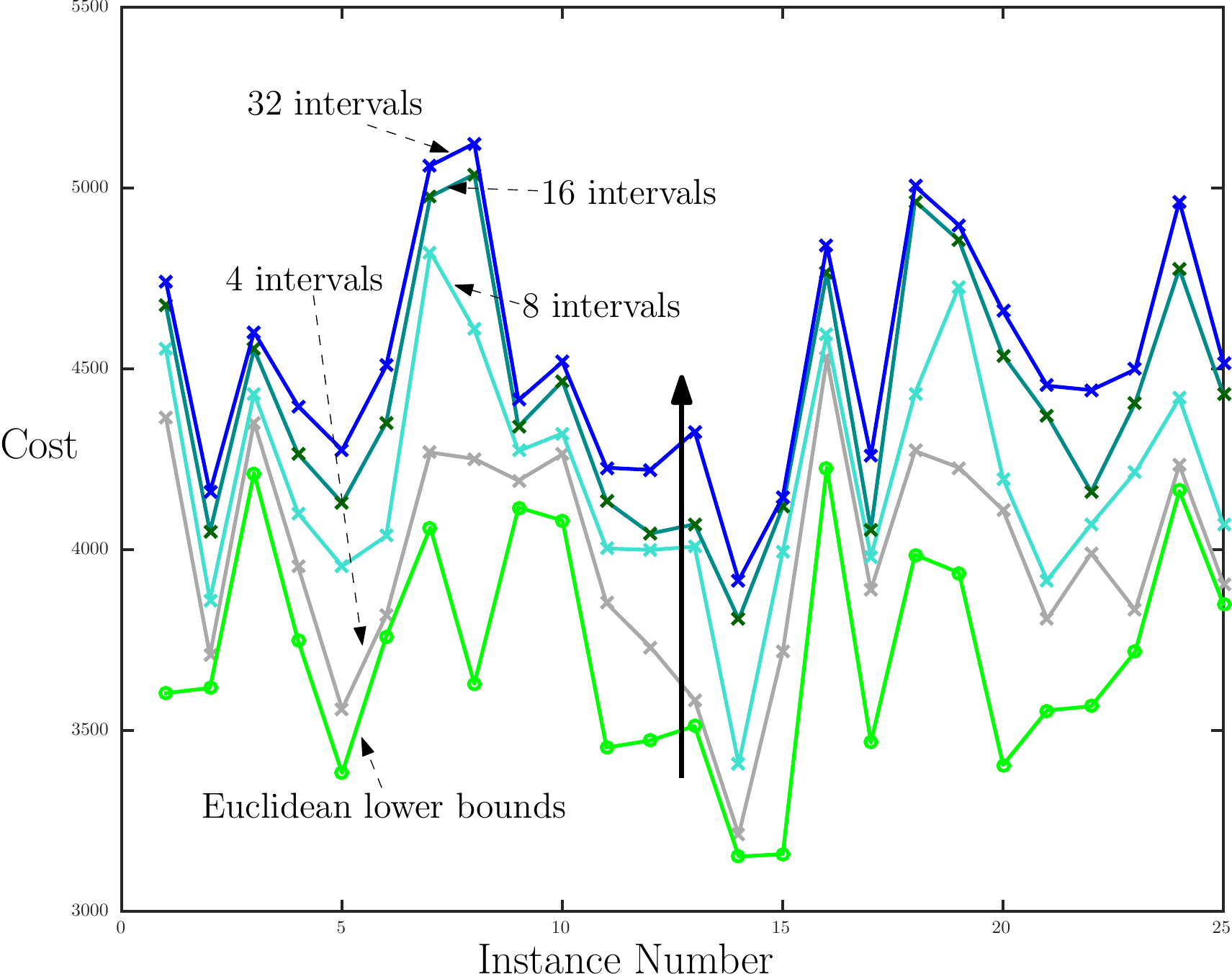}
\caption{Lower bounds computed with 4, 8, 16, and 32 intervals at each target for 25 instances.}
\label{fig:lb_prog}
\vspace{.75in}
\end{figure}

\begin{figure}[!h]
\centering{}
\includegraphics[width=4in]{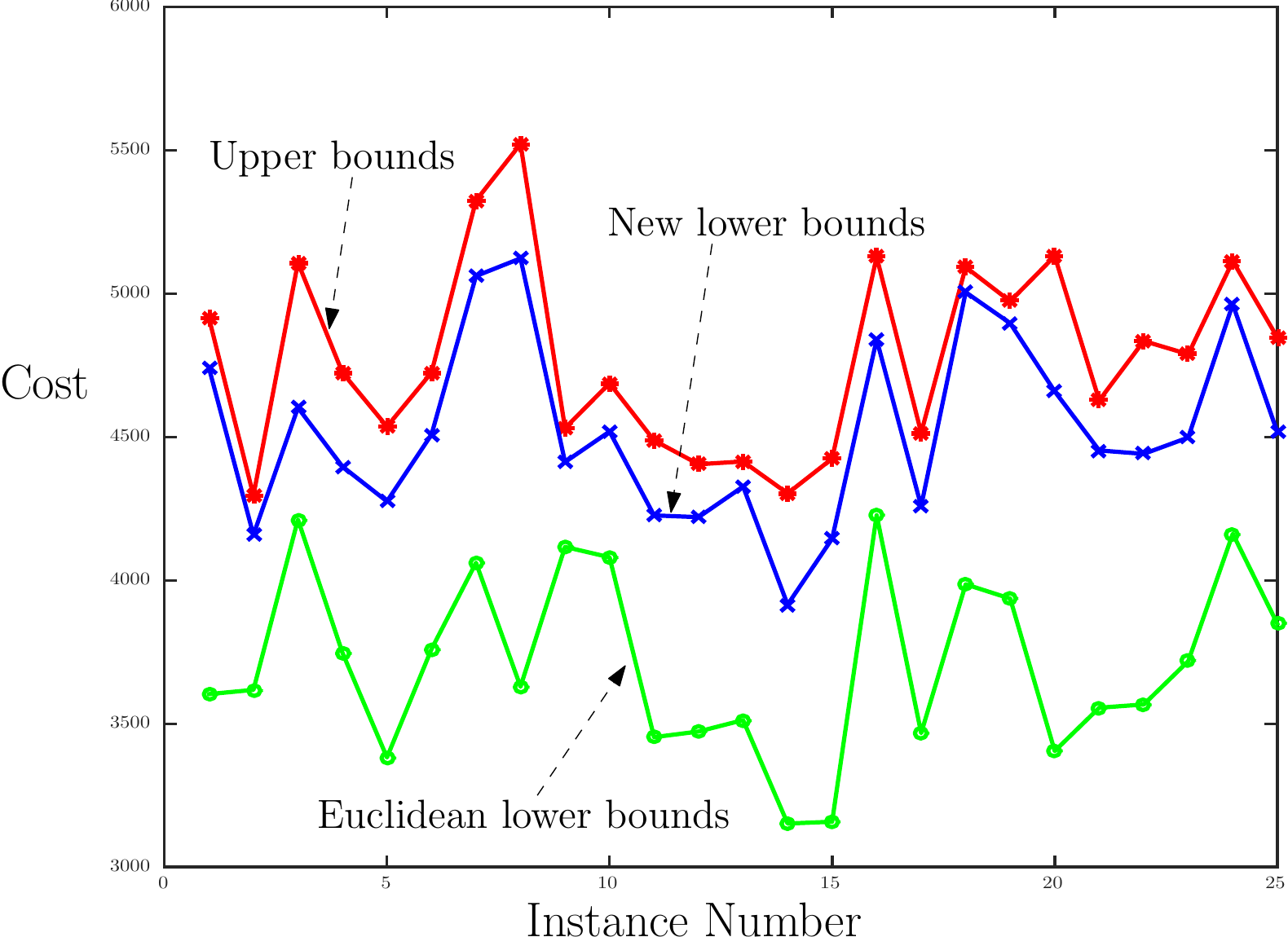}
\caption{Comparison between lower bounds and upper bounds for 32 discretizations, along with the optimal Euclidean TSP cost.}
\label{fig:best_ublb}
\vspace{.5in}
\end{figure}

\begin{figure}[!h]
\centering{}
\includegraphics[width=4in]{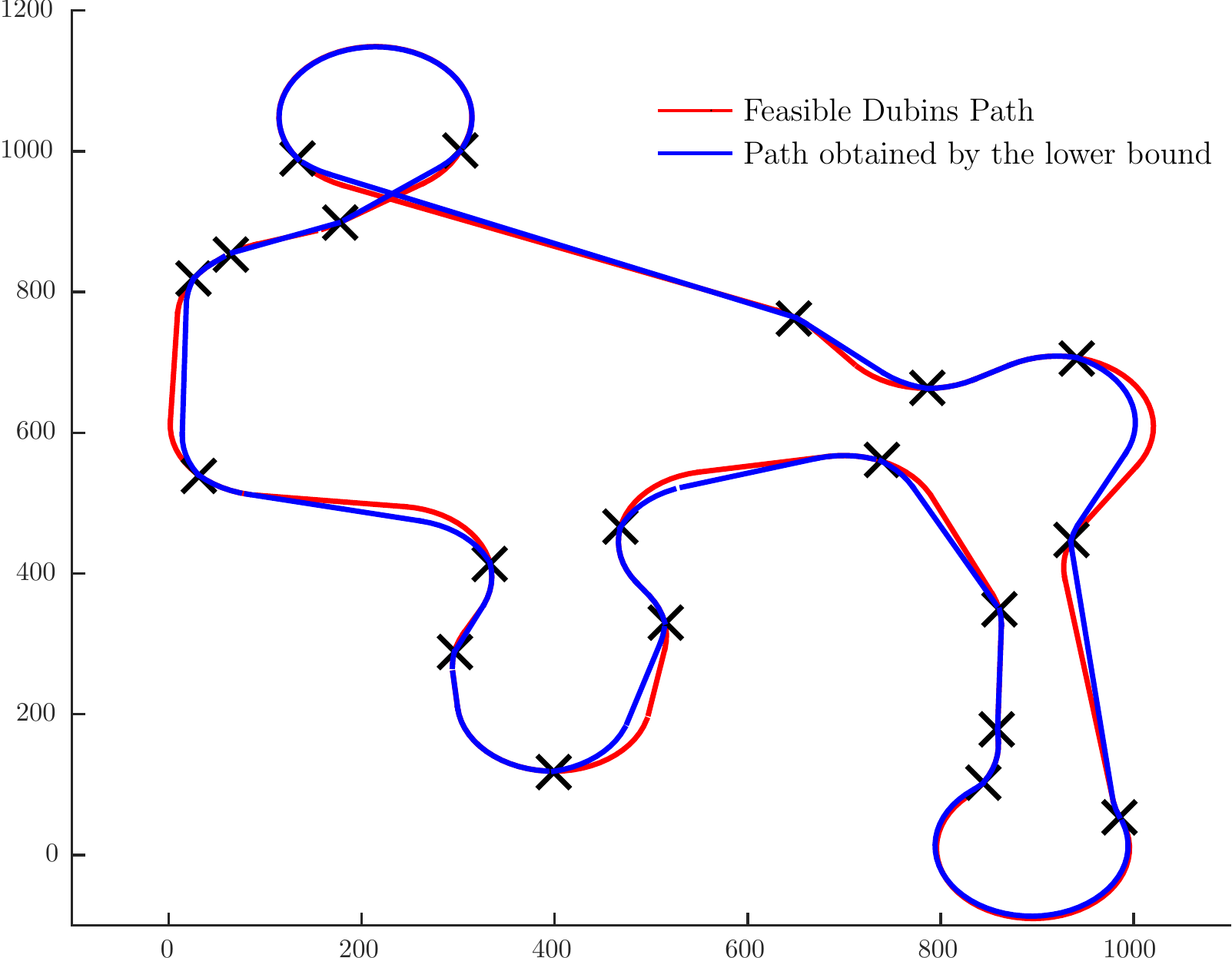}
\caption{A feasible Dubins path for an instance with 20 targets, and the path obtained from lower bound computation.}
\label{fig:paths}
\end{figure}

\section{Conclusion}

We provide a systematic procedure to find lower bounds for the DTSP. This article provides a new direction for developing approximation algorithms for the DTSP. Currently, the transformation method increases the size of the one-in-a-set TSP by 2 or 3 times, resulting in a large TSP. Computationally, more efficient tools for directly solving the one-in-a-set TSP will be useful in finding tighter lower and upper bounds for the DTSP. Future work can also address the same problem with multiple vehicles and other precedence constraints.

\bibliographystyle{IEEEtran}
\bibliography{NSFbib}

\section{Appendix}

\begin{figure}
        \centering
        \subfigure[$\bar{x}>2\rho$]{\includegraphics[scale=.54]{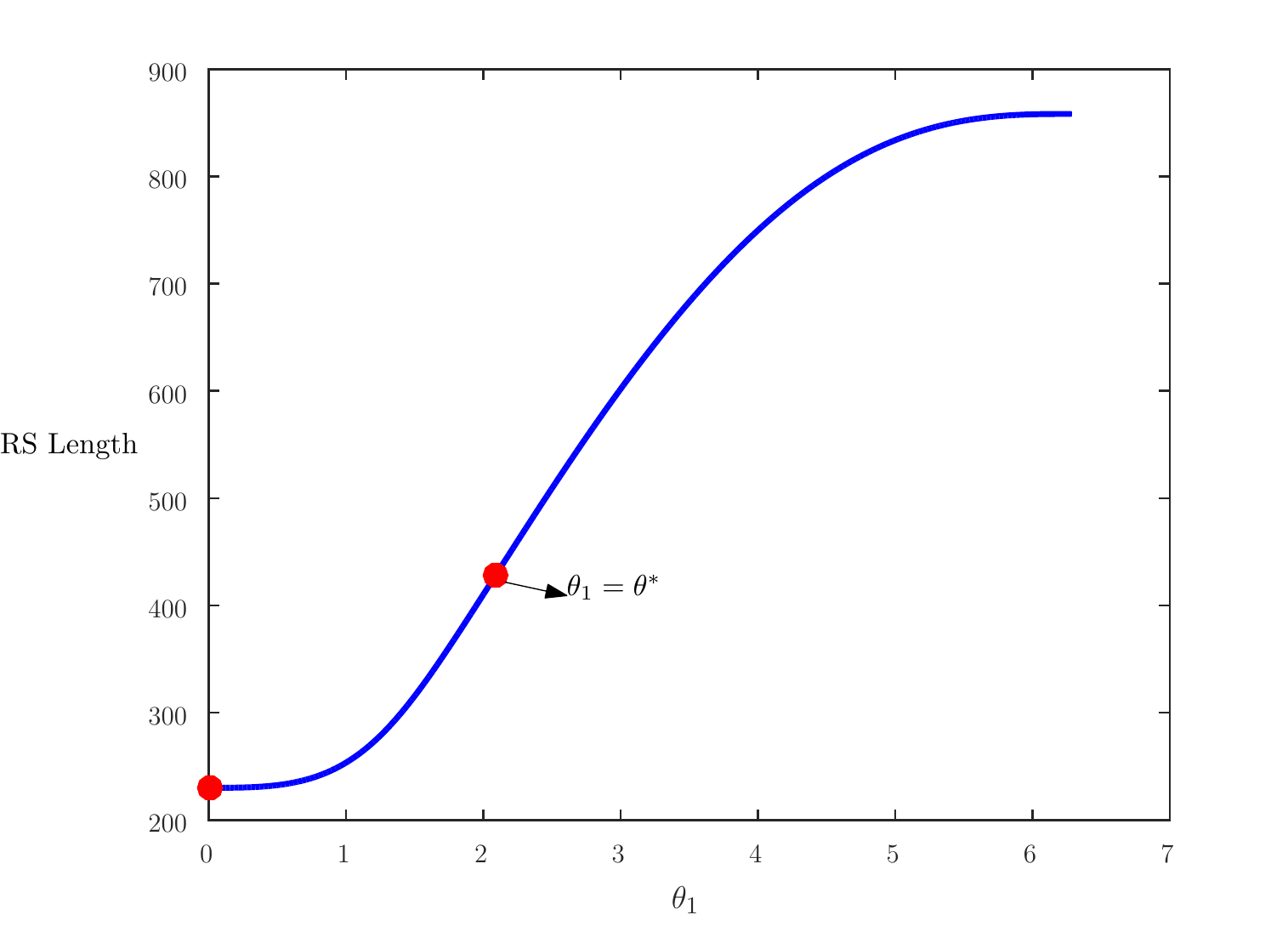}}
        \subfigure[$\bar{x}>2\rho$]{\includegraphics[scale=.54]{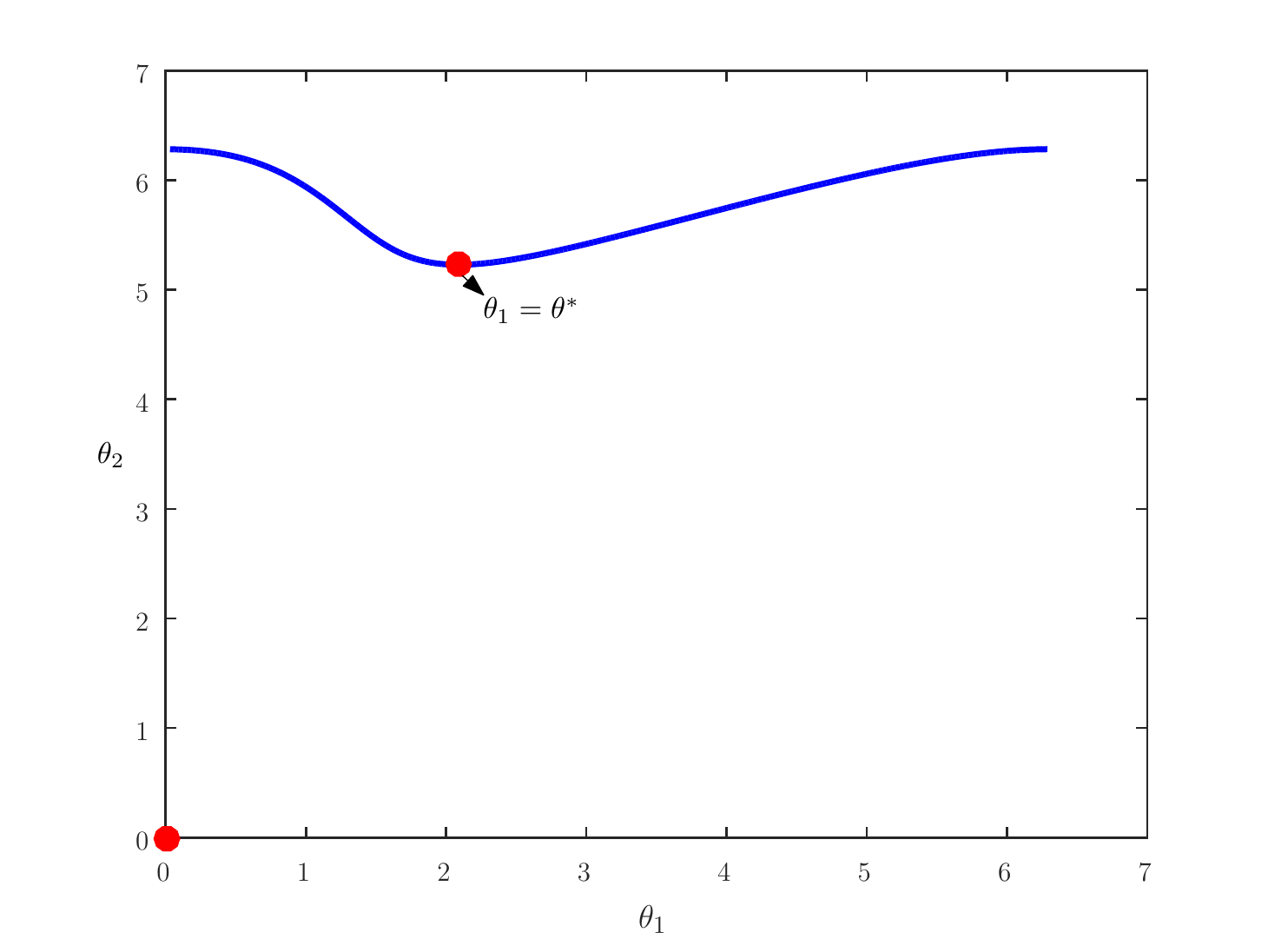}} \\
        \subfigure[$\bar{x}\leq 2\rho$]{\includegraphics[scale=.54]{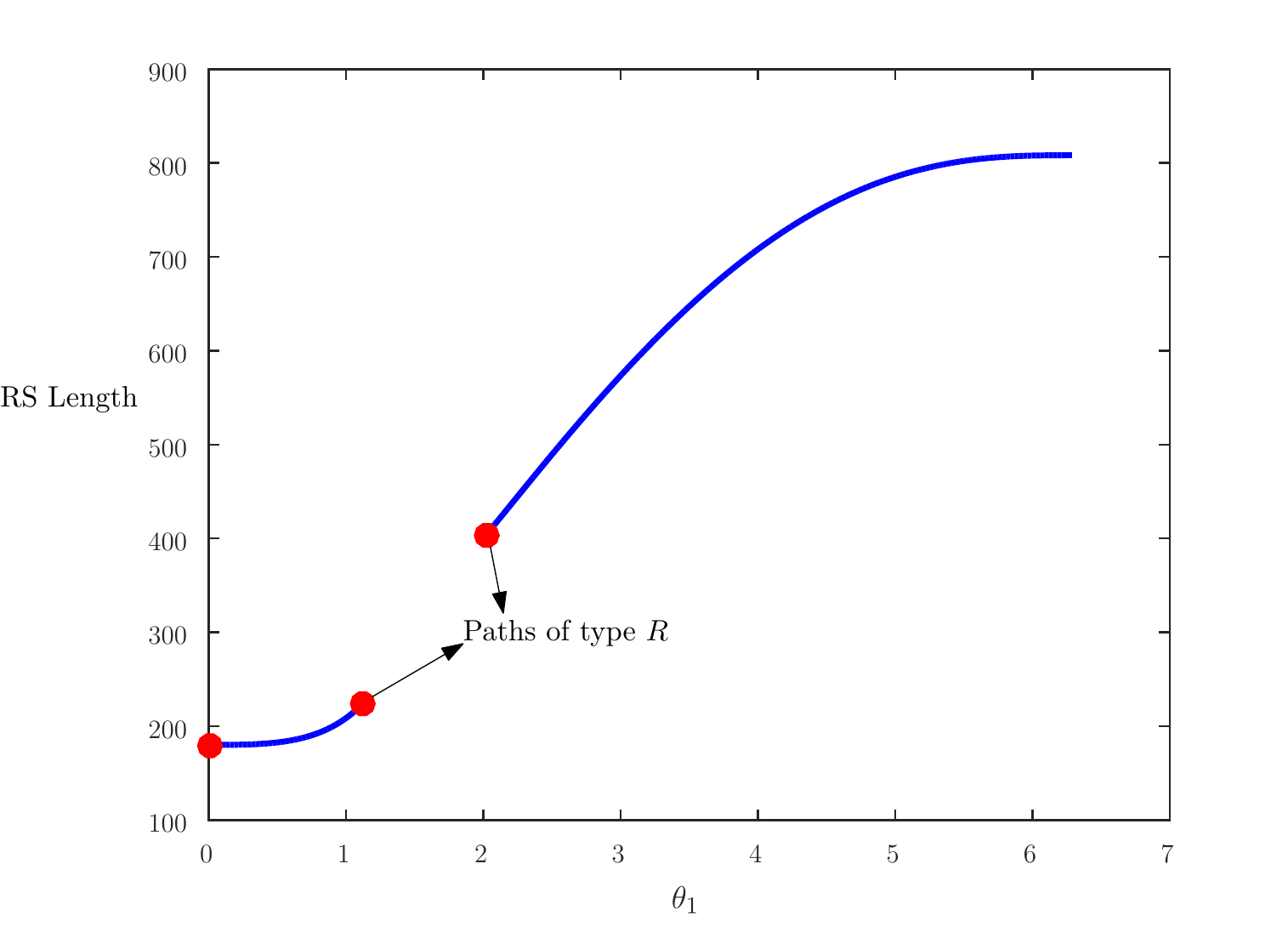}}
        \subfigure[$\bar{x}\leq 2\rho$]{\includegraphics[scale=.54]{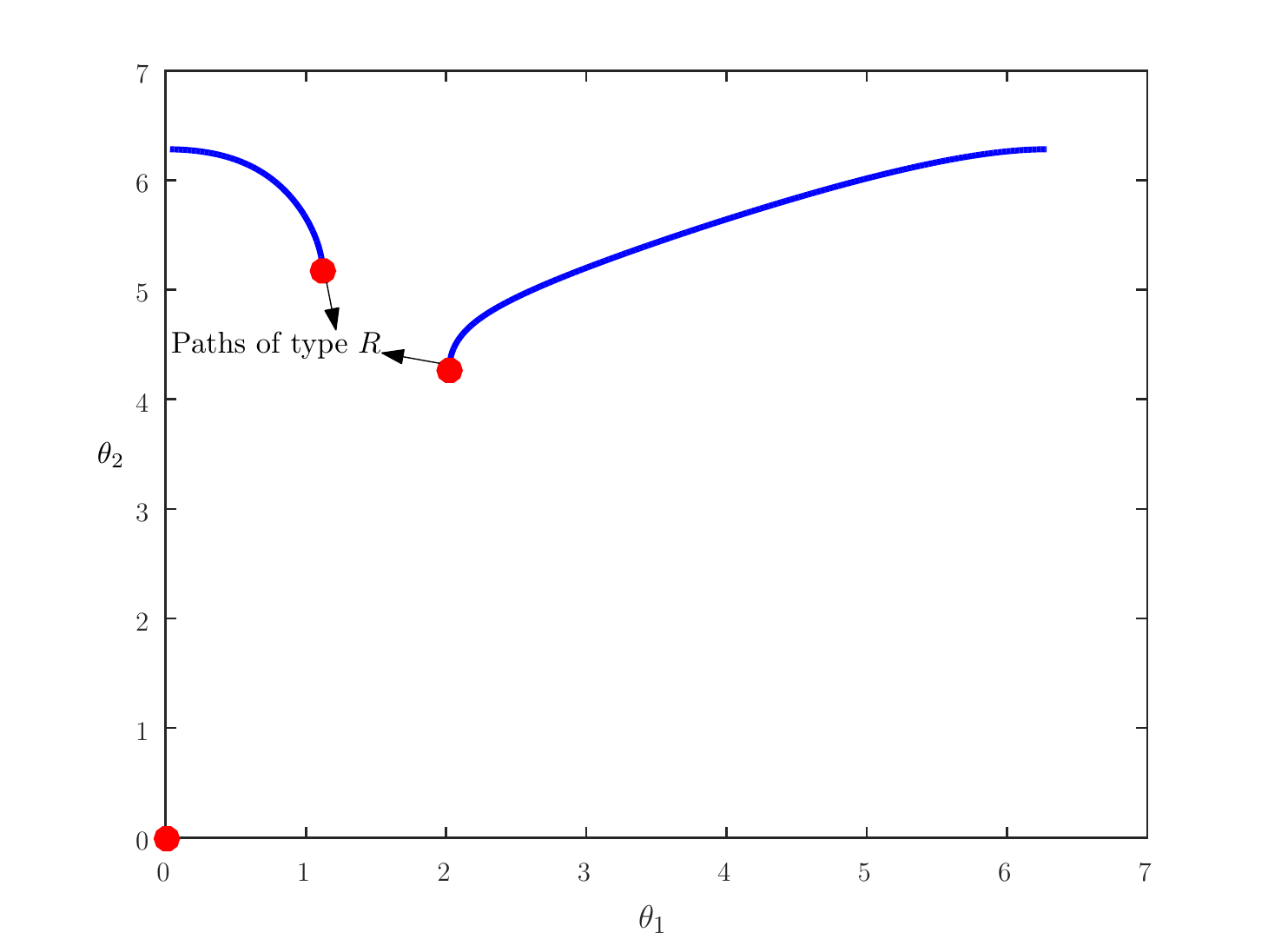}}
        \caption{$RS$ path: examples illustrating ${\mathfrak{D}}(\theta_1)$ and ${\theta_2}(RS,\theta_1)$.}
        \label{fig:distRS_d2r}
\end{figure}

\subsection{Proof of Lemma \ref{lemma:RS}}

 Using Fig. \ref{fig:RSpath}, one can relate $L$ and $\phi$ to $\theta_1$ using the following equations:
\begin{align}
\rho\sin\phi + L \cos\phi & = \bar{x} - \rho \sin\theta_1, \nonumber \\
\rho\cos\phi -L\sin\phi & = \rho \cos\theta_1. \label{RSeq}
\end{align}
 The arrival angle $\theta_2(RS,\theta_1)$ at target 2 is equal to $2\pi-\phi$. We now consider two different cases: $\bar{x}>2\rho$ and $\bar{x}\leq 2\rho$ (the $RS$ path does not exist for a subset of angles of $\theta_1$ if $\bar{x}<2\rho$). \\

\textbf{Case 1:} $\bar{x}> 2\rho$. \\

The length of the $RS$ path is ${\mathfrak{D}}:= (\theta_1+\phi)\rho + L$. Therefore, $\frac{d{\mathfrak{D}}}{d\theta_1}:= (1+\frac{d\phi}{d\theta_1})\rho + \frac{dL}{d\theta_1}$. The derivatives of $\phi$ and $L$ with respect to $\theta_1$ can be obtained by differentiating \eqref{RSeq} as follows:

\begin{align}
(\rho\cos\phi-L\sin\phi) \frac{d\phi}{d\theta_1} + \cos\phi\frac{dL}{d\theta_1} & = -\rho\cos\theta_1,\\
-(\rho\sin\phi + L \cos\phi) \frac{d\phi}{d\theta_1} -\sin \phi \frac{dL}{d\theta_1} & = -\rho\sin\theta_1.
\end{align}

Solving these equations and simplifying further, we obtain the following:
\begin{align}
 \frac{d\phi}{d\theta_1} & = \frac{\bar{x}}{L}\cos\phi -1, \\
 \frac{dL}{d\theta_1} & = -\frac{\rho\bar{x}}{L} \cos\theta_1.
\end{align}

Therefore,
\begin{align}
\frac{d{\mathfrak{D}}}{d\theta_1} &= (1+\frac{d\phi}{d\theta_1})\rho + \frac{dL}{d\theta_1}\\
                                  &= \frac{\bar{x}}{L}(\rho\cos\phi-\rho\cos\theta_1) \\
                                  &= \bar{x}\sin\phi.
\end{align}
For any $\theta_1\in [0,2\pi]$, it is easy to verify geometrically that $\phi\in [0,\pi]$ using Fig. \ref{fig:RSpath}. Therefore, $\forall \theta_1\in (0,2\pi)$, $\frac{d{\mathfrak{D}}}{d\theta_1} > 0$, i.e., the length of the $RS$ path increases monotonically from $\bar{x}$. When $\theta_1=2 \pi$, the curved segment in the $RS$ path vanishes and the length of the $RS$ path returns to the Euclidean distance between the targets ($\bar{x}$). Even though the length of the $RS$ path increases monotonically for any $\theta_1\in [0,2\pi)$, the arrival angle at target 2, $\theta_2:=2\pi-\phi$, first decreases with $\theta_1$, reaches a minimum at some $\theta_1=\theta^*$, and increases to $2\pi$. This minimum can be computed by solving $\frac{d\phi}{d\theta_1} =0 \Rightarrow \frac{\bar{x}}{L}\cos (\phi(\theta^*)) -1 = 0$ or $\cos(\phi(\theta^*)) = \frac{L}{\bar x}$. One can verify that at $\theta_1=\theta^*$, $\theta_2$ reaches a minimum.

Now, the optimum for $\min_{\theta_{1} \in I_1}\{RS^1(\theta_{1})\}$ must satisfy one of the following conditions:
\begin{enumerate}
\item $\frac{d{\mathfrak{D}}}{d\theta_1} = 0$ or $\theta_1=0$ ($\frac{d{\mathfrak{D}}}{d\theta_1}$ does not exist at this point) or $\theta_1=\theta_1^{min}$ or $\theta_1=\theta_1^{max}$. $\forall \theta_1\in (0,2\pi)$, $\frac{d{\mathfrak{D}}}{d\theta_1} \neq 0$. As the length of the $RS$ path increases monotonically with respect to $\theta_1$, we need not consider $\theta_1=\theta_1^{max}$. Therefore, for this condition, the optimum occurs when $\theta_1=0$ or $\theta_1=\theta_1^{min}$.

\item $\theta_2=\theta_2^{min}$ or $\theta_2=\theta_2^{max}$. 
\end{enumerate}

Therefore, when $\bar{x} > 2\rho$, $\min_{\theta_{1} \in I_1}\{RS^1(\theta_{1})\}:=\min \{ d_S, RS^1(\theta_1^{min}), RS^2(\theta_2^{min}), RS^2(\theta_2^{max})\}$.\\

\textbf{Case 2:} $\bar{x} \leq 2\rho$. \\

In this case, the $RS$ path is not defined for any $\theta_1\in (\sin(\frac{\bar{x}}{2\rho }),\frac{\pi}{2} + \cos(\frac{\bar{x}}{2\rho}))$. Moreover, when $\theta_1 = \sin(\frac{\bar{x}}{2\rho })$ or $\theta_1 = \frac{\pi}{2} + \cos(\frac{\bar{x}}{2\rho})$, the $RS$ path reduces to just one segment of type $R$. Therefore, following the same analysis as in the previous case, $\min_{\theta_{1} \in I_1}\{RS^1(\theta_{1})\}:=\min \{ d_S, d_R, RS^1(\theta_1^{min}), RS^2(\theta_2^{min}), RS^2(\theta_2^{max})\}$. Hence this case is proved.

\subsection{Proof of Lemma \ref{lemma:RL}}

We can solve for $\phi$ and $\theta_2$ using the following equations (Fig. \ref{fig:RLpath}):

\begin{align}
2\rho \cos\phi - \rho \cos \theta_2 & = \rho \cos\theta_1, \nonumber \\
2\rho \sin\phi + \rho \sin\theta_2 & = \bar{x} - \rho \sin\theta_1.\label{RLeq}
\end{align}

Differentiating and simplifying these equations, we get

\begin{align}
-2 \sin \phi \frac{d\phi}{d\theta_1} + \sin \theta_2 \frac{d\theta_2}{d\theta_1} & = -\sin \theta_1, \\
2 \cos \phi \frac{d\phi}{d\theta_1} + \cos \theta_2 \frac{d\theta_2}{d\theta_1} & = -\cos \theta_1.
\end{align}

Solving further for $\frac{d\phi}{d\theta_1}$ and $\frac{d\theta_2}{d\theta_1}$, we get
 \begin{align}
 \frac{d\phi}{d\theta_1} & = \frac{\sin (\theta_1-\theta_2)}{2\sin (\phi+\theta_2)}, \\
 \frac{d\theta_2}{d\theta_1} & = -\frac{\sin (\theta_1+\phi)}{\sin (\phi+\theta_2)},  \\
\frac{d{\mathfrak{D}}}{d\theta_1} &= \rho(1+\frac{d\theta_2}{d\theta_1}+2\frac{d\phi}{d\theta_1}) \nonumber \\
                                  &= \rho(1- \frac{\sin (\theta_1+\phi)}{\sin (\phi+\theta_2)} + \frac{\sin (\theta_1-\theta_2)}{\sin (\phi+\theta_2)} ).
\end{align}

Equating $\frac{d{\mathfrak{D}}}{d\theta_1}=0$ and simplifying the equations, we get either $\phi + \theta_1 = 0$ or $\phi + \theta_2 = 0$ or $\theta_1=\theta_2$. $\phi + \theta_1 = 0$ or $\phi + \theta_2 = 0$ would imply that one of the circles vanishes; however, this is possible only when $\bar{x}\leq 2\rho$. When $\theta_1=\theta_2$, we note that $\frac{d\theta_2}{d\theta_1} = -1$ and $\frac{d\phi}{d\theta_1} = 0$. Using this, one can verify that  $\frac{d^2{\mathfrak{D}}}{d\theta_1^2} = \frac{2(1-\cos(\theta_1+\phi))}{\sin(\theta_1+\phi)}$ $\Rightarrow$ $\frac{d^2{\mathfrak{D}}}{d\theta_1^2}> 0$. Therefore, the length of the $RL$ path reaches a minimum when $\theta_1=\theta_2$. \\

\textbf{Case 1:} $4\rho \geq \bar{x} \geq 2\rho$. \\

The optimum for $\min_{\theta_{1} \in I_1}\{RL^1(\theta_{1})\}$ must occur at one of the extreme values of ${\mathfrak{D}}(\theta_1)$ or when $\theta_1\in\{\theta_1^{min},\theta_1^{max}\}$ or $\theta_2\in\{\theta_2^{min},\theta_2^{max}\}$. ${\mathfrak{D}}(\theta_1)$ reaches a local minimum at $\theta_1=\theta^{1*}$ (Fig. \ref{fig:distRL}). Also, the $RL$ path just ceases to exist when $\theta_1 =\theta^{2*}$ or $\theta_1 =\theta^{3*}$. Specifically, for a small $\epsilon > 0$, the $RL$ path does not exist when $\theta_1 =\theta^{2*}-\epsilon$ or $\theta_1 =\theta^{3*}+\epsilon$. Therefore, $\min_{\theta_{1} \in I_1}\{RL^1(\theta_{1})\}:=\min \{RL^1(\theta^{1*}), RL^1(\theta^{2*}), RL^1(\theta^{3*}), RL^*\}$. \\

\textbf{Case 2:} $2\rho \geq \bar{x} \geq 0$. \\

In this case, one of the circles may cease to exist, and therefore the optimum may be equal to $d_L$ or $d_R$ if the corresponding angle constraints are met. Following the same arguments as in the previous case, we obtain $\min_{\theta_{1} \in I_1}\{RL^1(\theta_{1})\}:=\min \{d_L,d_R,RL^1(\theta^{1*}), RL^*)\}$. Hence this case is proved.

\begin{figure}
        \centering
        \subfigure[$\bar{x}>2\rho$]{\includegraphics[scale=.46]{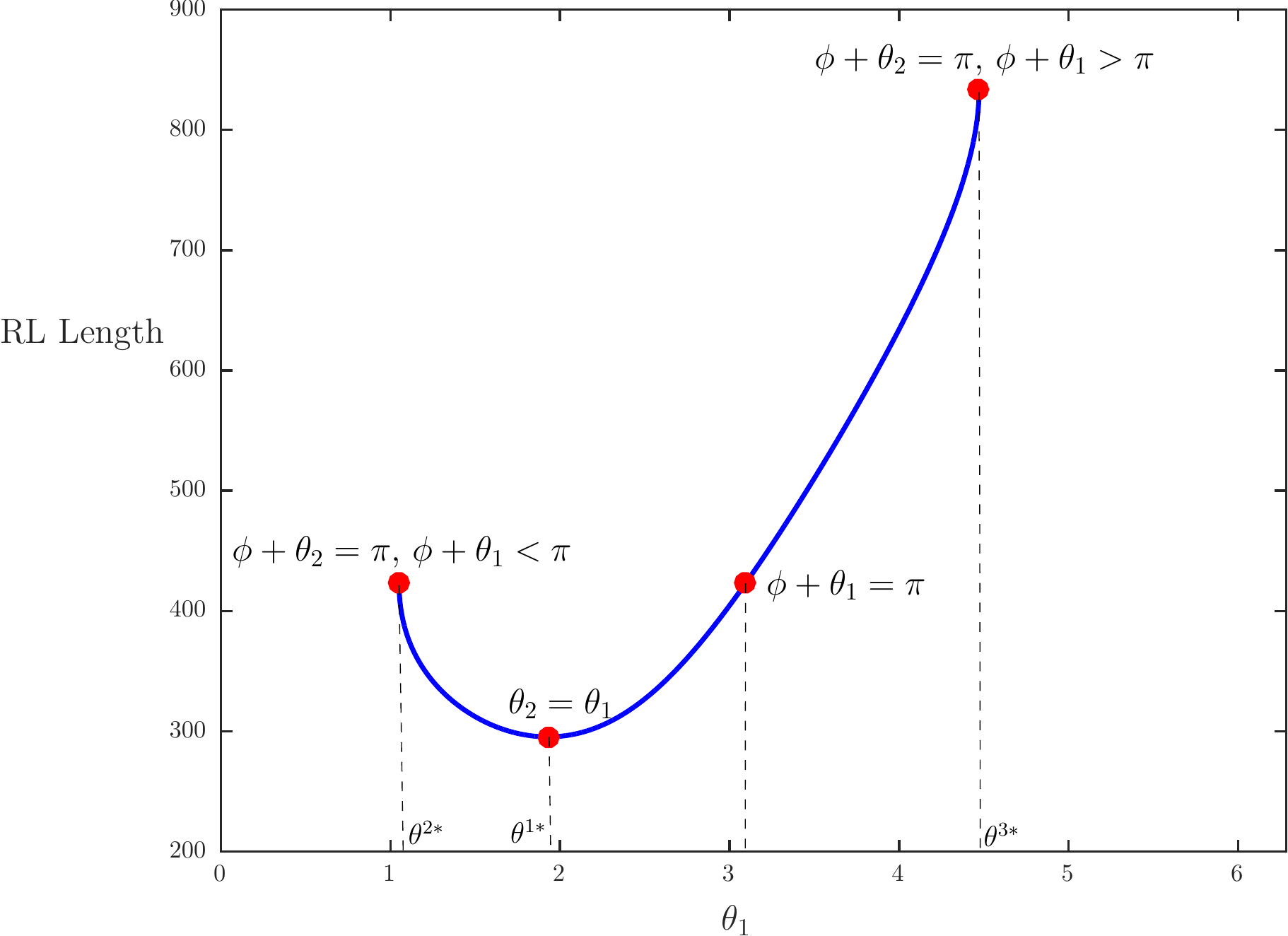}}
        \subfigure[$\bar{x}>2\rho$]{\includegraphics[scale=.46]{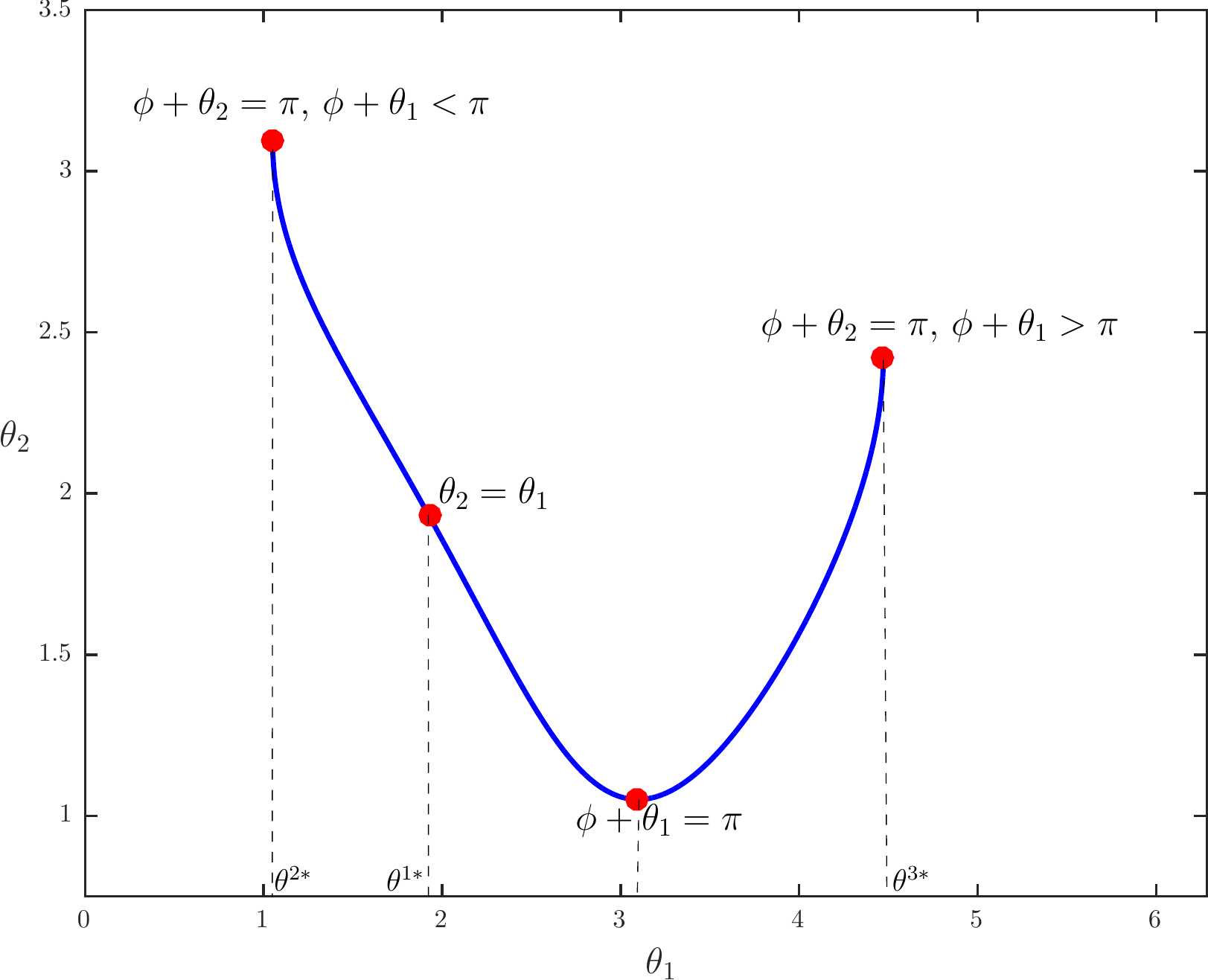}} \\
        \subfigure[$\bar{x}\leq 2\rho$]{\includegraphics[scale=.46]{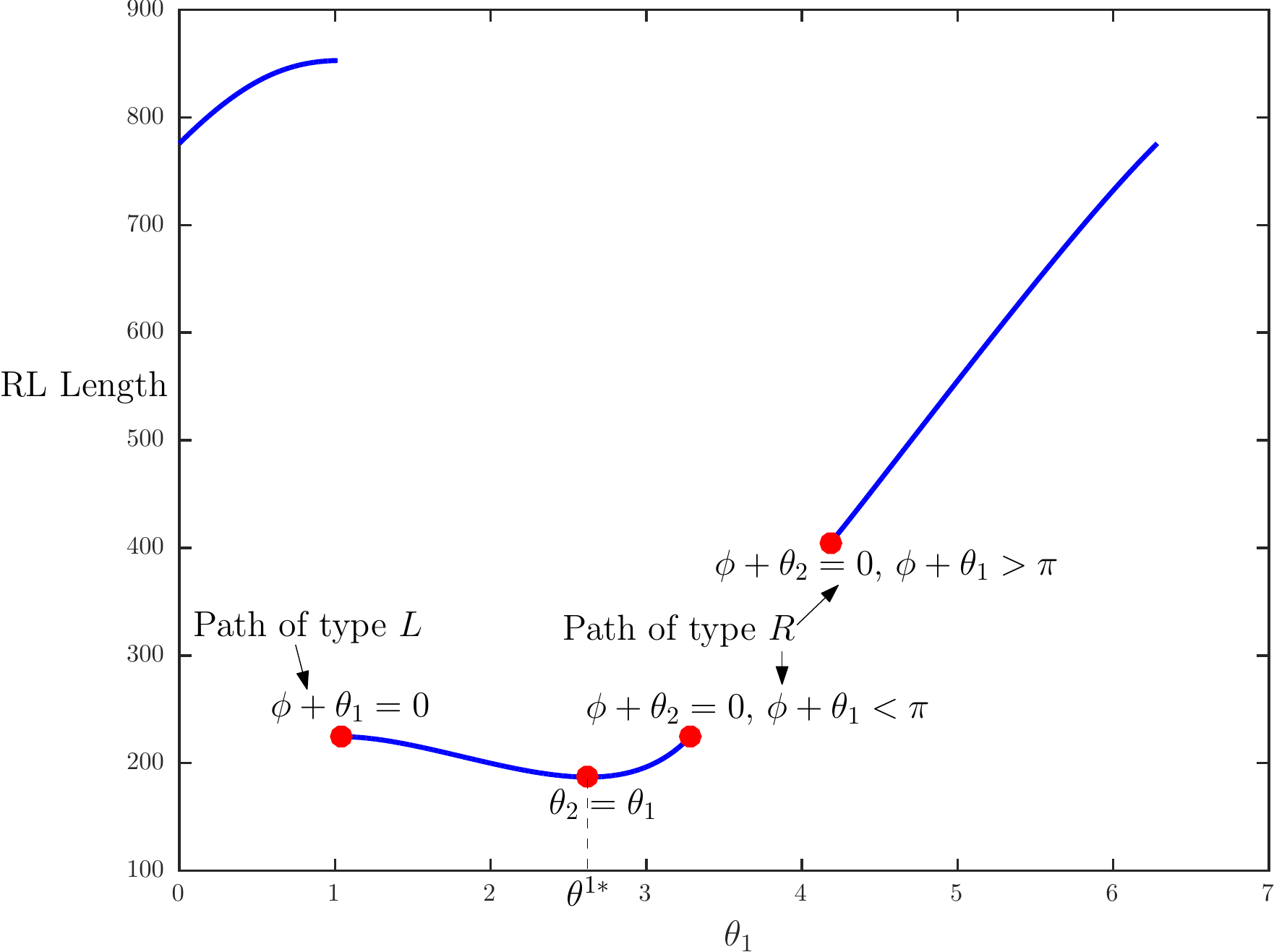}}
        \subfigure[$\bar{x}\leq 2\rho$]{\includegraphics[scale=.46]{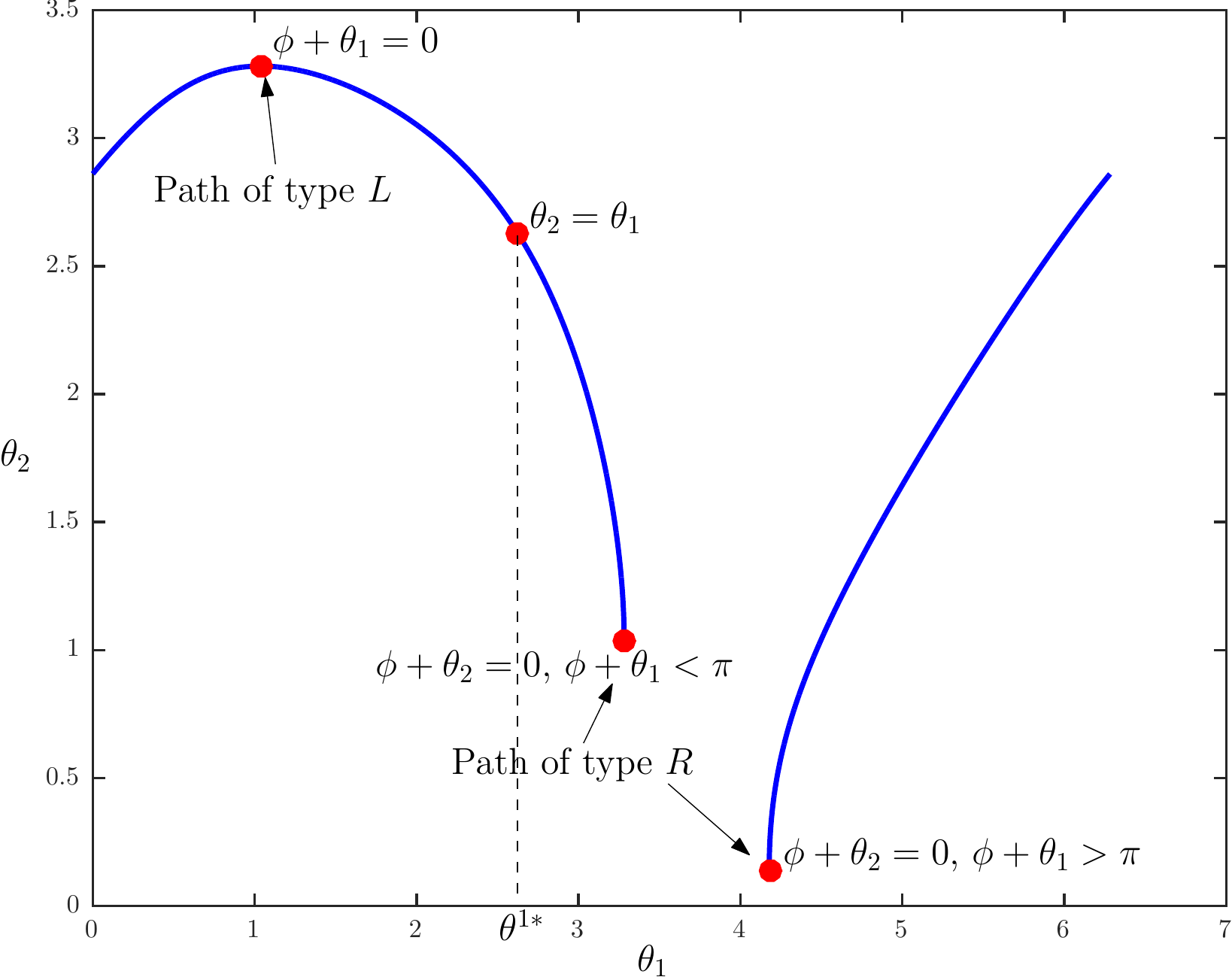}}
        \caption{$RL$ Path: examples illustrating ${\mathfrak{D}}(\theta_1)$ and ${\theta_2}(RL,\theta_1)$ for the case when $0\leq \phi+\theta_2 \leq \pi$.}
        \label{fig:distRL}
\end{figure}

\end{document}